\documentclass{article}
\pdfoutput=1
\usepackage{amsmath,amsthm,amssymb}
\usepackage[utf8]{inputenc}
\usepackage{hyperref}
\usepackage{arydshln} 
\usepackage{mathtext}
\usepackage[T1,T2A]{fontenc}
\usepackage{tabularx}
\usepackage[english]{babel}
\usepackage{ragged2e}

\usepackage[square,numbers]{natbib}
\usepackage{graphicx}
\usepackage{stmaryrd}
\usepackage{dsfont}
\usepackage[all]{xy}
\usepackage{gensymb}
\usepackage{url}
\usepackage{indentfirst}
\usepackage{booktabs}
\usepackage{geometry}
\usepackage{array}
\usepackage{comment}

\usepackage{float}

\usepackage{tkz-euclide}
\usetikzlibrary{calc}
\setcitestyle{square}

\RequirePackage{
,tikz
,verbatim
,amsmath
,amssymb
,amsthm
,pgfplotstable
,geometry
,graphicx
,pdfpages
,pgfplots
,refstyle
,hyperref
,csquotes
,indentfirst
,multicol
,enumerate
}

\newcommand{\N}{\mathbb{N}}

\newcommand{\R}{\mathbb{R}}

\pgfplotsset{compat=1.14}
\makeatletter
\def\@fnsymbol#1{\ensuremath{\ifcase#1\or *\or 
  **\or \dagger\dagger
  \or \ddagger\ddagger \else\@ctrerr\fi}}
\makeatother

\newtheorem{theorem}{Theorem}
\newtheorem{lemma}[]{Lemma}
\newtheorem{proposition}[]{Proposition}
\newtheorem{corollary}{Corollary}
\newtheorem{open}{Open problem}
\newtheorem{conjecture}{Conjecture}

\theoremstyle{definition}
\newtheorem{definition}[]{Definition}
\newtheorem{remark}{Remark}

\def\vs{\vskip 0.30cm}

\title{\large \bfseries{On graphs with unique geoodesics and antipodes}}
\author{\normalsize Dmitriy Gorovoy\thanks{Jagiellonian University, Krakow, Poland}  \ and David Zmiaikou\thanks{Harbour Space University, Barcelona, Spain} }
\date{}

\begin{document}

\colorlet{best}{red!70!blue!110!violet!110!white!100}


\maketitle

\begin{abstract}
In 1962, Oystein Ore asked in which graphs there is exactly one geodesic between any two vertices. He called such graphs \emph{geodetic}. In this paper, we systematically study properties of geodetic graphs, and also consider \emph{antipodal} graphs, in which each vertex has exactly one antipode (a farthest vertex). We find necessary and sufficient conditions for a graph to be geodetic or antipodal, obtain results related to algorithmic construction, and  find interesting families of Hamiltonian geodetic graphs. By introducing and describing the maximal hereditary subclasses and the minimal hereditary superclasses of the geodetic and antipodal graphs, we get close to the goal of our research -- a constructive classification of these graphs.
\end{abstract}
 
 {\raggedleft Key words: geodetic graphs, antipodal graphs, Hamiltonian geodetic graphs, bearing trees, hereditary subclasses and superclasses}
 


\begin{justify}


\section*{Introduction}

In 1962, in his famous book on graph theory \cite{ore}, the Norwegian mathematician Oystein Ore presented a new problem: ``In a tree there is a unique shortest arc between any two vertices, but there are also other connected graphs with the same property. Try to characterize these \emph{geodetic} graphs in other way.``

\vs
\begin{definition} 
A \emph{geodesic} between two vertices of a graph is a shortest path
connecting these vertices. 
\end{definition} 

\begin{definition} 
A graph is called \emph{geodetic} if for any pair of its vertices there is exactly one geodesic between them.
\end{definition} 

{ \raggedright See examples of geodetic graphs in the Figure \ref{fig:semeistva}.}

{\raggedright \noindent\textbf{Analogy with geometry.}
As we know, in the Euclidean geometry of a space of zero curvature ($\R^n$), as well as in the geometry of Lobachevsky spaces
of negative curvature (hyperbolic half-plane and Poincaré disc) between any two points there is exactly one geodesic. Geodetic graphs have a discrete analogue of this property.}







   

\begin{figure}[htb]
\centering
\begin{tikzpicture}[scale=0.7,line width=0.03cm,
block/.style={circle,fill=best,draw=best,inner sep=1.25pt}]
    \foreach \i in {1,...,5}
    {
        \node[block] (\i) at (\i*360/5:1.5) {};
    }
    \begin{scope}[color=best]
\draw (1) --(2);
\draw (1) --(3);
\draw (1) --(4);
\draw (1) --(5);
\draw (3) --(2);
\draw (4) --(3);
\draw (5) --(4);
\draw (2) --(5);
\draw (4) --(2);
\draw (5) --(3);
    \end{scope}
    \node[color=black] at (-90:2) {Complete graphs};
\end{tikzpicture}
\quad
\begin{tikzpicture}[scale=0.7,color=best, x=0.6cm,y=0.35cm, line width=0.02cm]
\draw  (0,2)-- (-2,0);
\draw  (0,2)-- (0, 0);
\draw  (0,2)-- (2, 0);
\draw  (-1,-2)-- (0, 0);
\draw  (1,-2)-- (2, 0);
\draw  (3,-2)-- (2, 0);
\draw  (3,-2)-- (4, -4);
\draw  (3,-2)-- (2, -4);

\draw [fill=best] (0,2) circle (2pt);
\draw [fill=best] (-2,0) circle (2pt);
\draw [fill=best] (2,0) circle (2pt);
\draw [fill=best] (1,-2) circle (2pt);
\draw [fill=best] (-1,-2) circle (2pt);
\draw [fill=best] (3,-2) circle (2pt);
\draw [fill=best] (0,0) circle (2pt);
\draw [fill=best] (2, -4) circle (2pt);
\draw [fill=best] (4, -4) circle (2pt);
 \node[color=black] at (-80:6) {Trees};
\end{tikzpicture}
\quad
\begin{tikzpicture}[scale=0.7,line width=0.03cm,
block/.style={circle,fill=best,draw=best,inner sep=1.25pt}]
    \foreach \i in {1,...,7}
    {
        \coordinate (\i) at (\i*360/7:1.5);
        \node[block] at (\i*360/7:1.5) {};
    }
    \begin{scope}[color=best]
\draw (1) --(2);
\draw (3) --(2);
\draw (3) --(4);
\draw (5) --(4);
\draw (5) --(6);
\draw (7) --(6);
\draw (7) --(1);
    \end{scope}
    \node[color=black] at (-90:2) {Odd cycles};
\end{tikzpicture}
\quad
\begin{tikzpicture}[scale=0.7,x=0.6cm,y=0.6cm,line width=0.03cm,
block/.style={circle,fill=best,draw=best,inner sep=1.25pt}]
    \foreach \i in {1,...,5}
    {
        \coordinate (\i) at (\i*360/5:1.5);
        \node[block] at (\i *360/5:1.5) {};
    }
    \begin{scope}[color=best]
      \coordinate (6) at (129:3);
        \node[block] at (129:3) {};
              \coordinate (7) at (159:3);
        \node[block] at (159:3) {};
             \coordinate (8) at (114:4.5);
        \node[block] at (114:4.5) {};
              \coordinate (9) at (129:5);
        \node[block] at (129:5) {};
            \coordinate (10) at (144:4.5);
        \node[block] at (144:4.5) {};
                      \coordinate (11) at (15:3);
        \node[block] at (15:3) {};
            \coordinate (12) at (-15:3);
        \node[block] at (-15:3) {};
                      \coordinate (13) at (45:3);
        \node[block] at (45:3) {};
            \coordinate (14) at (30:4.5);
        \node[block] at (30:4.5) {};
           \coordinate (15) at (-45:3);
        \node[block] at (-45:3) {};
            \coordinate (16) at (-30:4.5);
        \node[block] at (-30:4.5) {};
        
\draw (1) --(2);
\draw (3) --(2);
\draw (3) --(4);
\draw (5) --(4);
\draw (5) --(1);
\draw (2) --(6);
\draw (2) --(7);
\draw (7) --(6);
\draw (8) --(6);
\draw (6) --(9);
\draw (6) --(10);
\draw (9) --(8);
\draw (5) --(11);
\draw (5) --(12);
\draw (11) --(12);
\draw (11) --(13);
\draw (11) --(14);
\draw (12) --(16);
\draw (12) --(15);
\draw (15) --(16);
    \end{scope}
        \node[color=black] at (-80:3) {Cacti without even cycles};
\end{tikzpicture}
\caption{Examples of geodetic graphs.}
\label{fig:semeistva}
\end{figure}
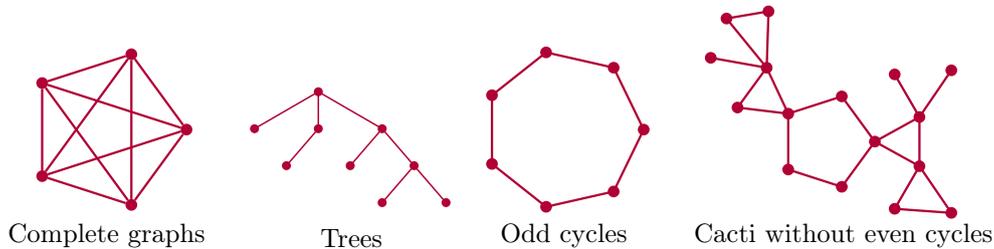

\noindent\textbf{Applications.} Studying the connections between vertices of a graph and finding criteria for the existence of a configuration of these connections is of exceptional interest in network modeling. In particular, geodetic graphs are applied in the design of computer systems and networks (see \cite{frasser}).

\noindent In contrast to the uniqueness of geodesics, the following question naturally arises: in which graphs does each vertex have exactly one antipode?

\begin{definition}
An \emph{antipode} of a given vertex is the vertex farthest from the given one. An \emph{antipodal graph} is a connected graph in which each vertex has exactly one antipode.
\end{definition}

{\raggedright See examples of antipodal graphs in the Figure \ref{fig:semeistva monoant}.}

{\raggedright \noindent\textbf{Analogy with geometry.}
In the spherical geometry of a space of positive curvature $\mathbb{S}^n$ between two points there can be infinitely many geodesics, but each point has exactly one antipode. Discrete analogs here are antipodal graphs.}




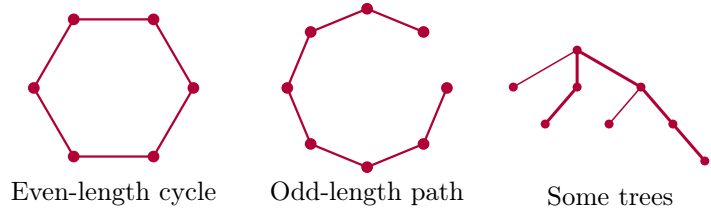
\begin{figure}[htb]
\centering
\begin{tikzpicture}[scale=0.7,color=best,line width=0.03cm,
block/.style={circle,fill=best,draw=best,inner sep=1.25pt}]
    \foreach \i in {1,...,6}
    {
        \coordinate (\i) at (\i*360/6:1.5);
        \node[block] at (\i*360/6:1.5) {};
    }
\begin{scope}[color=best]
\draw (1) --(2);
\draw (1) --(6);
\draw (3) --(2);
\draw (4) --(3);
\draw (5) --(4);
\draw (5) --(6);
\end{scope}
\node[color=black] at (-90:2) {Even-length cycle};
\end{tikzpicture}
\quad
\begin{tikzpicture}[scale=0.7,color=best,line width=0.03cm,
block/.style={circle,fill=best,draw=best,inner sep=1.25pt}]
    \foreach \i in {1,...,8}
    {
        \coordinate (\i) at (\i*360/8:1.5);
        \node[block] at (\i*360/8:1.5) {};
    }
    \begin{scope}[color=best]
\draw (1) --(2);
\draw (3) --(2);
\draw (3) --(4);
\draw (5) --(4);
\draw (5) --(6);
\draw (7) --(6);
\draw (7) --(8);
    \end{scope}
    \node[color=black] at (-90:2) {Odd-length path};
\end{tikzpicture}
\quad
\begin{tikzpicture}[scale=0.7,color=best, x=0.6cm,y=0.35cm, line width=0.02cm]
\draw  (0,2)-- (-2,0);
\draw[line width=0.04cm] (0,2)-- (0, 0);
\draw[line width=0.04cm]  (0,2)-- (2, 0);
\draw[line width=0.04cm] (-1,-2)-- (0, 0);
\draw  (1,-2)-- (2, 0);
\draw[line width=0.04cm] (3,-2)-- (2, 0);
\draw[line width=0.04cm] (3,-2)-- (4, -4);

\draw [fill=best] (0,2) circle (2pt);
\draw [fill=best] (-2,0) circle (2pt);
\draw [fill=best] (2,0) circle (2pt);
\draw [fill=best] (1,-2) circle (2pt);
\draw [fill=best] (-1,-2) circle (2pt);
\draw [fill=best] (3,-2) circle (2pt);
\draw [fill=best] (0,0) circle (2pt);
\draw [fill=best] (4, -4) circle (2pt);
 \node[color=black] at (-80:6) {Some trees};
\end{tikzpicture}
\caption{Examples of antipodal graphs.}
\label{fig:semeistva monoant}
\end{figure}

\vs
\noindent\textbf{Goal.} The main goal of our research is to describe geodetic and antipodal graphs explicitly, \emph{i.e.} to find their constructive classifications which would allow us to build them algorithmically.

\vs
\noindent\textbf{Main results.} In this article we obtain the following:
\begin{itemize}
    \item a criterion for any graph with some condition on one of its spanning trees to be geodetic (Theorem \ref{th:stebli});
    \item a criterion for a tree to be antipodal (Proposition \ref{pro: monoant tree});
    \item two infinite series of Hamiltonian geodetic graphs of diameters two and four (Theorems \ref{th: family1} and \ref{th:dim 2 gam});
    \item a criterion for any graph to be geodetic, $K_{1,3}$-free (Theorem \ref{th: K_1,3});
     \item algorithms of polynomial complexity $O(n^3)$ for checking whether a weighted graph is geodetic/antipodal or not (Theorem \ref{th: vzv geod});
    \item an algorithm for assigning positive weights to the edges of an arbitrary connected graph so that it becomes both geodetic and antipodal (Theorem \ref{th: vzveshen});
    \item the maximal hereditary subclass of the  class of geodetic graphs (Theorem \ref{th: FIS(P_)}) ;
    \item the maximal hereditary subclass and the minimal hereditary superclass of the class of antipodal graphs (Theorem \nolinebreak\ref{th: ant her});
    \item the minimal hereditary superclasses of the classes of weighted geodetic and antipodal graphs (Theorem \ref{th: weighted her} for geodetic and  Proposition \ref{pro: any graph in antipodal} for antipodal graphs).
    \end{itemize}

\section{Geodetic and antipodal graph check and bearing trees}
\begin{lemma}\label{le:monogeo_algo}
Let $G$ be a connected graph on $n$ vertices. Then it is possible to check whether the graph $G$ is geodetic/antipodal or not for $O(n^3)$. 
\end{lemma}
\begin{proof}
Let $v_1, v_2, ... v_n$ be vertices of the graph $G$. We will check whether $G$ is geodetic or not at the steps of constructing the spanning tree $T$ of the graph $G$. Our algorithm, in fact, is exactly Breadth-first search. Take the vertex $v_1$ as the root of the tree; it constitutes zero tier. Further, add to $T$ all vertices that are adjacent to $v_1$  (together with the edges from $v_1$). Let WLOG these vertices are $v_2, v_3, ... v_k$; they form tier 1.

Next, add all the neighbors of the vertex $v_2$ that have not been considered earlier; they will form tier 2.

Now we consider the neighbors of the vertex $v_3$ in the graph $G$. Two options for the next step of the algorithm:
\begin{enumerate}
    \item[1)] If some of the neighbors $ v_3 $ was not added earlier, add it to tier 2 of the tree $T$.
    \item[2)] If some of the neighbors of $ v_3 $ is also a neighbor of $ v_2 $ from tier 2, then we conclude that the graph is not geodetic.
\end{enumerate}

Further, we analogously consider other vertices of the graph $G$. At the end of the algorithm, we find out if there is exactly one geodesic from $v_1$ to all other vertices.

To build such an algorithm, $O(n^2)$ is required. Repeating this algorithm, choosing each vertex as the root of the tree, we need $O(n^3)$.
\end{proof}

\begin{definition} \label{de:opornoye}
 A \emph{bearing tree} of the graph \nolinebreak $G$ is the tree $T$, built in the proof of Lemma \ref{le:monogeo_algo}. It's always true for any bearing tree $T$ that the distance between the root of the tree and any vertex of the $k$-th tier is equal to $k$ not only in the tree, but also in the graph $G$ itself. (For an arbitrary spanning tree, this is not always true, since the distance between vertices of the graph $G$ can be less than the distance between these vertices in its spanning tree).
\end{definition}

\section{Geodetic graphs} 
\subsection{Blocks of geodetic graphs}
\begin{definition} \label{de:peregorodka}
A \emph{balk} of the bearing tree $T$ is an edge of the graph $G$ with the ends lying in the same tier of the  tree \nolinebreak$T$ (see Figure \ref{fig: def diag}). 
\end{definition}
\begin{figure}[h]
\centering
\begin{tikzpicture}[scale=0.8,color=best,x=0.8cm,y=0.8cm,line width=0.04cm,
block/.style={circle,fill=best,draw=best,inner sep=1.5pt}]
    \foreach \i in {1,...,6}
    {
        \coordinate (\i) at (\i*360/6:2);
        \node[block] at (\i*360/6:2) {};
    }
    \draw (1)--(2)--(3)--(4)--(5)--(6)--(1);
    \draw (2)--(4);
    \draw (3)--(6);
        \filldraw ($(3)!0.5!(6)$) circle (2pt);
        \coordinate (7) at ($(3)!0.5!(6)$);
        \begin{scriptsize}
        \node[color=black, scale=1.25, anchor=south west] at (1) {$1$}; 
        \node[color=black, anchor=south east, scale=1.25] at (2) {$2$}; 
        \node[color=black, anchor=east,scale=1.25] at (3) {$3$}; 
       \node[color=black, anchor=north east, scale=1.25] at (4) {$4$}; 
        \node[color=black, anchor=north west, scale=1.25] at (5) {$5$}; 
        \node[color=black, anchor=west,scale=1.25] at (6) {$6$}; 
         \node[color=black, anchor=south,scale=1.25] at (7) {$7$}; 
         \end{scriptsize}
 \end{tikzpicture}
 \hspace{1cm}
\begin{tikzpicture}[scale=0.8,color=best, x=0.7cm,y=0.5cm, line width=0.02cm]

\draw (0,0) -- (-2, -2)--(-1,-4);
\draw (-2,-2)--(-3,-4);
\draw (0,0) -- (2, -2)--(1,-4) ;
\draw(2,-2)--(3,-4);
\draw[dotted] (-3,-4) arc (-180:0:3 and 1.5);
\draw [dotted] (-3,-4)--(-1,-4);
\draw[dotted] (1,-4)--(-1,-4);
\draw [fill=best] (0,0) circle (2pt);
\draw [fill=best] (-2,-2) circle (2pt);
\draw [fill=best] (2, -2) circle (2pt);
\draw [fill=best] (-1, -4) circle (2pt);
\draw [fill=best] (-3,-4) circle (2pt);
\draw [fill=best] (1,-4) circle (2pt);
\draw [fill=best] (3,-4) circle (2pt);
        \begin{scriptsize}
        \node[color=black, scale=1.25, anchor=south] at (0,0) {$1$}; 
        \node[color=black, anchor=south east, scale=1.25] at (-2,-2) {$2$}; 
        \node[color=black, anchor=east,scale=1.25] at (-3,-4) {$3$}; 
        \node[color=black, anchor=north, scale=1.25] at (-1,-4) {$4$}; 
        \node[color=black, anchor=north, scale=1.25] at (1,-4) {$5$}; 
        \node[color=black, anchor=south west,scale=1.25] at (2,-2) {$6$};
        \node[color=black, anchor=west,scale=1.25] at (3,-4) {$7$};  
         
        \node[color=black, scale=1.25] at (-5,-4) {$2$}; 
        \node[color=black,  scale=1.25] at (-5,-2) {$1$}; 
        \node[color=black, scale=1.25] at (-5,0) {$0$}; 
        \end{scriptsize}
\end{tikzpicture}
\caption{A bearing tree with three balks.}
\label{fig: def diag}
\end{figure}
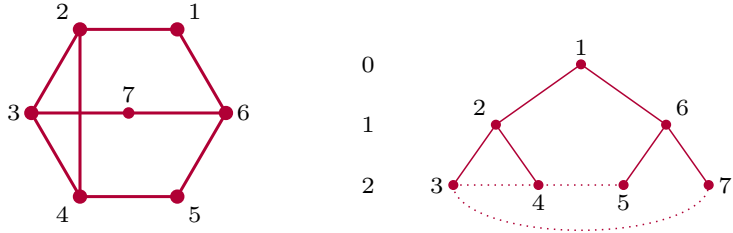

\begin{lemma} \label{le:opornoe}
Let $G$ be a geodetic graph, $T$ its bearing tree. Then any edge $e \in E(G) \setminus E(T)$ is a balk of $T$.
\end{lemma}
\begin{proof}
All ends of edges not added to $T$ lie either in the same tier or between adjacent tiers. Indeed, let the vertex $a$ from the tier $l$ be adjacent to the vertex $b$ from the $l+i $ of the tier, where $ i \ge 0$. Then the distance in the graph $G$ between the root and the vertex $b$ does not exceed $l+1$, and on the other hand, by the definition of the bearing tree, it is equal to $l+i$. Therefore, $i$ is equal to 0 or 1. If $i=1$, then the graph is not geodetic - there are two geodesics between the root and the vertex $b$. Hence, $i=0$.
\end{proof}

\begin{definition} \label{def: stem}
Let $T$ be some tree with a root $u$.
A \emph{stem} of a tree $T$ is a path whose ends are a root $u$ and an arbitrary leaf of the tree. Note that two stems can intersect.
\end{definition}
\vspace{-0.2cm}

For example, the tree in the Figure \ref{fig: pobeg} has five stems (one of length 1, two of length 2, and two of length \nolinebreak3).

\begin{figure}[h]
\centering
\begin{tikzpicture}[color=best, x=0.6cm,y=0.35cm, line width=0.02cm]
\draw  (0,2)-- (-2,0);
\draw  (0,2)-- (0, 0);
\draw  (0,2)-- (2, 0);
\draw  (-1,-2)-- (0, 0);
\draw  (1,-2)-- (2, 0);
\draw  (3,-2)-- (2, 0);
\draw  (3,-2)-- (4, -4);
\draw  (3,-2)-- (2, -4);
\draw [fill=best] (0,2) circle (2pt);
\draw [fill=best] (-2,0) circle (2pt);
\draw [fill=best] (2,0) circle (2pt);
\draw [fill=best] (1,-2) circle (2pt);
\draw [fill=best] (-1,-2) circle (2pt);
\draw [fill=best] (3,-2) circle (2pt);
\draw [fill=best] (0,0) circle (2pt);
\draw [fill=best] (2, -4) circle (2pt);
\draw [fill=best] (4, -4) circle (2pt);
\node[color=black, above] at (0,2) {$u$};
\node[color=black, below] at (-2,0){$1$};
\node[color=black, below] at (-1,-2) {$2$};
\node[color=black, below] at (1,-2) {$3$};
\node[color=black, below] at (2,-4) {$4$};
\node[color=black, below] at (4,-4) {$5$};
\end{tikzpicture}
\caption{A tree with five stems.}
\label{fig: pobeg}
\end{figure}
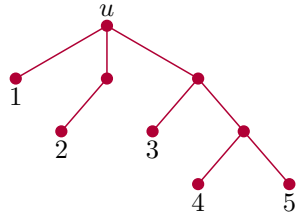

\begin{lemma} \label{le: dva rebra mezdu dvumya vetvyami}
Let $G$ be a geodetic graph and $T$ be its bearing tree. Then there is at most one balk between any two stems of the tree $T$.
\end{lemma}

\begin{proof}
Suppose there are two balks between some two stems of the tree $T$ as in the Figure \ref{fig: 2 edges betw vetv i potok}.
Consider a pair of vertices $a$ and $d$, there are two paths of length $l-m+1$ between them (in an even cycle $abdca$). But then, since the graph is geodetic, there must be a path of length at most $l-m$ between the vertices $a$ and $d$. Therefore, there are two geodesics from the root of the tree to the vertex $d$ (one of them passes through $a$).
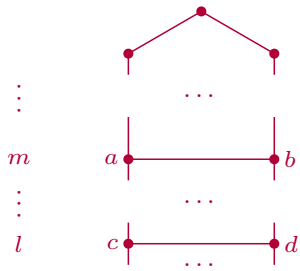
\begin{figure}[htb]
\centering
\begin{tikzpicture}[scale=0.8,color=best, x=0.6cm,y=0.35cm, line width=0.02cm]
\draw  (0,0)-- (-2,-2);
\draw  (0,0)-- (2, -2);
\draw  (-2,-3)-- (-2,-2);
\draw  (2,-3)-- (2, -2);

\draw  (-2,-8)-- (-2, -5);
\draw  (2,-8)-- (2, -5);
\draw  (2,-7)-- (-2, -7);
\draw  (2,-11)-- (-2, -11);
\draw  (-2,-12)-- (-2, -10);
\draw  (2,-12)-- (2, -10);

\node at (0,-4) {$\dots$};
\node at (0,-9) {$\dots$};
\node at (0,-12) {$\dots$};
\node[rotate=90] at (-5,-9) {$\dots$};
\node[rotate=90] at (-5,-4) {$\dots$};

\draw [fill=best] (0,0) circle (2pt);
\draw [fill=best] (-2,-2) circle (2pt);
\draw [fill=best] (2,-2) circle (2pt);
\draw [fill=best] (2, -7) circle (2pt);
\draw [fill=best] (-2,-7) circle (2pt);

\draw [fill=best] (2, -11) circle (2pt);
\draw [fill=best] (-2,-11) circle (2pt);

\begin{scriptsize}
\node[left][scale=1.25] at (-2, -7) {$a$};
\node[right][scale=1.25] at (2, -7) {$b$};
\node[left][scale=1.25] at (-2, -11) {$c$};
\node[right][scale=1.25] at (2, -11) {$d$};
\node[scale=1.25]  at (-5, -11) {$l$};
\node[scale=1.25] at (-5, -7) {$m$};
\end{scriptsize}
\end{tikzpicture}
\caption{Two balks.}
\label{fig: 2 edges betw vetv i potok}
\end{figure}
\end{proof}

\begin{lemma} \label{le: K_n v pervom yaruse}

Let $G$ be a geodetic graph. Then, fixing an arbitrary vertex $v$ of the graph $G$ as the root of the bearing tree, we get that all vertices of the first tier of the bearing tree are split into several complete subgraphs of the graph $G$, perhaps consisting of one vertex.
\end{lemma}
\begin{proof}
Assume the statement of the lemma is incorrect, that is there is some non-complete connected subgraph $H$ of the graph $G$ in the first tier of the bearing tree with the root $v$. This means that  there is such chain $u_1u_2u_3$ in $H$ that $u_1u_3 \notin E$. But then  there are two geodesic lengths of two between the vertices $u_1$ and $u_3$: $u_1v u_3$ and $u_1u_2u_3$. Thus we got a contradiction.
\end{proof}

\begin{definition}
A \emph{cut vertex} of a graph $G$ is a vertex of $G$ such that $G - x$ is a disconnected graph. A \emph{block} of a graph is a maximal connected subgraph with no cut vertex – a subgraph with as many edges as possible and no cut vertex.
\end{definition}

\begin{lemma}[Stemple and Watkins \cite{stemple-watkins}] \label{le: block or lobe}
A graph $G$ is geodetic if and only if each of its blocks is geodetic.
\end{lemma}

\begin{theorem} \label{th: K_1,3}
Graph $G$ is geodetic, $K_{1,3}$-free if and only if $G$ is either an odd-length cycle or a graph, each block of which is complete and each vertex of which belongs to at most two blocks.
\end{theorem}
\begin{proof}
Consider an arbitrary block $H$ of the graph $G$. If the degree of some vertex from $H$ is less than two, then $H$ is $K_2$ or $K_1$, otherwise there would be a cut vertex in the block. In the case when the degree of each vertex of the given block is equal to two, we get that $H$ is an odd cycle (it cannot be an even cycle according to the Lemma \ref{le: block or lobe} because even cycle is not a geodetic graph). Further, consider the case when $H$ contains a vertex of degree at least 3 and consider a bearing tree $T$ taking this vertex as a root. All vertices of the first tier are split into several complete subgraphs due to Lemma \ref{le: K_n v pervom yaruse}. Let these complete subgraphs be $K_{m_1}, ... K_{m_n}$. If $n \geq 3$, then there is $K_{1,3}$ as an induced subgraph of $G$.

 Let $n=2$, and the subgraphs in the first tier are $K_{m_1}$ and $K_{m_2}$ (note that at least one of the subgraphs contains at least two vertices). Let us call the set of all stems (See Definition \nolinebreak\ref{def: stem}) passing through $K_{m_1}$ the set $A$, and the set of all stems passing through $K_{m_2}$, respectively, as $B$. There is at least one balk between the stems $A$ and $B$, since they belong to the same block. Further, consider the uppermost such balk, let without loss of generality it be in the $k$-th tier between the stems $s_2$ from $A$ and $s_3$ from $B$, and let $m_1 \geq 2$ (see Figure \nolinebreak \ref{fig: K_1,3}). Then there are two different paths of length $k+1$ between the vertices $u$ and $w$ in the Figure \ref{fig: K_1,3} (one passes through the root of the tree $T$ and the other contains all edges of the stem $s_2$ between tier 1 and tier $k$ and two balks between $s_1$, $s_2$ and $s_2$, $s_3$). Therefore,  there must be a path of length at most $k$ between these vertices. On the other hand, the distance between the vertices of tiers $1$ and $k$ is at least $k-1$, and since the stems from $A$ and $B$ do not intersect at the vertices except for the root (by the Lemma \ref{le:opornoe}), there must be a balk between the stems. Therefore, the distance between $u$ and $w$ is at least $k$. And since there are no balks above the $k$-th tier, there is some stem $s_1$ passing through $u$, from which a balk goes to $s_3$ in the $k$-th tier. But then consider the induced subgraph of the graph $G$ with vertices $w, t_1, t_2, t_3$ (see Figure \nolinebreak \ref{fig: K_1,3}). Note that there can be no edges between $t_3$ and $t_1, t_2$ (by the Lemma \ref{le:opornoe}). There cannot be an edge $t_1t_2$ by the Lemma \nolinebreak\ref{le: dva rebra mezdu dvumya vetvyami}. Therefore, this subgraph is isomorphic to $K_{1,3}$.

\begin{figure}[H]
\centering
\begin{tikzpicture}[scale=0.9,color=best, x=0.6cm,y=0.25cm, line width=0.02cm]
\draw (0,-6) -- (2, -9);
\draw (0,-6) -- (-2, -9) -- (0,-9) ;
\draw (0, -11.5) -- (0,-9);
\draw (2,-11.5) -- (2, -9);
\draw (-2,-11.5) -- (-2, -9) ;
\draw (0, -14) -- (0,-17);
\draw (2,-14) -- (2, -17);
\draw (-2,-14) -- (-2, -17);
\draw (0, -6) -- (0,-9);
\draw[dotted] (-2,-17) arc (180:360:2 and 1);
\draw (0,-17) -- (2, -17);

\node at (0,-13) {$\dots$};
\node[color=black,rotate=90] at (-5,-13) {$\dots$};

\draw [fill=best] (0,-6) circle (2pt);
\draw [fill=best] (0,-9) circle (2pt);
\draw [fill=best] (-2, -9) circle (2pt);
\draw [fill=best] (2, -9) circle (2pt);
\draw [fill=best] (-2,-14) circle (2pt);
\draw [fill=best] (0,-14) circle (2pt);
\draw [fill=best] (2,-14) circle (2pt);
\draw [fill=best] (-2,-17) circle (2pt);
\draw [fill=best] (0,-17) circle (2pt);
\draw [fill=best] (2,-17) circle (2pt);

\begin{scriptsize}
    \node[color=black,below][scale=1.25] at (-2, -18.5) {$s_1$};
    \node[color=black,below] [scale=1.25]  at (0, -18.5) {$s_2$}; 
    \node[color=black,below] [scale=1.25]  at (2, -18.5) {$s_3$};
    \node[color=black,right] [scale=1.25]  at (2, -17) {$w$};
    \node[color=black,left] [scale=1.25]  at (-2, -17) {$t_1$};
    \node[color=black,left] [scale=1.25]  at (0, -17) {$t_2$};
     \node[color=black,right] [scale=1.25]  at (2, -14) {$t_3$};
    \node[color=black,right] [scale=1.25] at (2,-9) {$K_{m_2}$};
    \node[color=black,left] [scale=1.25] at (-2.5,-9) {$K_{m_1}$};
    \node[color=black,above] [scale=1.25] at (-2,-9) {$u$};
    \node[color=black,scale=1.25] at (-5, -17) {$k$};
    \node[color=black,scale=1.25]  at (-5, -9) {$1$};
\end{scriptsize}
\end{tikzpicture}
\caption{The case \nolinebreak$n=2$ .}
\label{fig: K_1,3}
\end{figure}
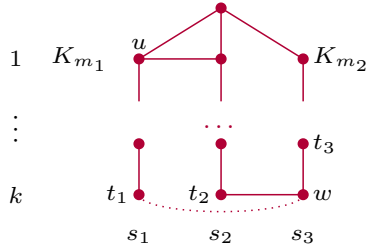


Thus, $n=1$ for some vertex $v$ of degree at least three. Then, fixing $v$ as the root of the bearing tree, we get that the block $H$ is complete. Indeed, all stems of the tree pass through some vertices of the first tier. But if there are some vertices in the second tier, then there is at least one balk between stems, passing through different vertices of the first tier (as $H$ is two-connected). But then we get contradiction with Lemma \ref{le: dva rebra mezdu dvumya vetvyami} as there is already a balk between these stems in the first tier.
Now, knowing what blocks can be, let's find out how they can intersect at vertices. An odd cycle on more than 3 vertices cannot intersect with any of the other blocks, otherwise $K_{1,3}$ will appear. It is also obvious that one vertex cannot belong to three blocks -- complete graphs. On the other hand, it is easy to check using Lemma \ref{le: block or lobe} that the graphs described in the statement of the theorem satisfy the condition and are geodetic and $K_{1,3}$-free.
\end{proof}


\begin{corollary}
Let some geodetic two-connected non-complete graph $G$ contains a clique on $m > 1$ vertices. Then the graph $G$ contains $K_{1,m}$ as an induced subgraph.
\end{corollary}
\begin{proof}
Let's fix some vertex $v$ of the clique $K_m$, which has a neighbor outside the clique, as the root of the bearing tree of the graph $G$. Consider the highest balk between stems containing at least one vertex of $K_m$ other than $v$ and a stem intersecting with $K_m$ only at the vertex $v$. Such balk exists as the graph $G$ is two-connected. Then the rest of the proof is analogous to the proof of the Theorem \ref{th: K_1,3}, where instead of $t_1$ and $t_2$ in the $k$-th tier there will be vertices $t_1, \dots t_{m-1}$.
\end{proof}

\begin{proposition}
Let $k$ be an even natural number and $G$ be some graph. Consider a graph $G(k)$ obtained from $G$ by adding $k$ vertices on each of its edges. A graph $G(k)$ is geodetic if and only if $G$ is geodetic.
\end{proposition}

\begin{proof} Sufficiency. Consider a geodetic graph $G$ and prove that $G(k)$ is also geodetic.
\begin{enumerate}
\item[(a)] Geodesics between old vertices from $G$: if the distance between two vertices in the graph $G$ was equal to $r$, it will become $(k+1)r$. Suppose that there is not one geodesic in $G(k)$ between corresponding vertices. Then there is also not one geodesic in the initial graph between the considered vertices.

\item[(b)] Geodesic between the new vertex $a$ and the old $b$: let $a$ vertex appear on the edge of $uv$, where $u$ and $v$ are the old vertices. Then geodesics from $b$ to $a$ pass either through $u$ or through $v$.
If through exactly one of them, then according to the previous subparagraph there is exactly one geodesic between $a$ and $b$.
Let there be two geodesics between the vertices and they pass through $u$ and $v$. Then the distance from $b$ to $u$ and the distance from $b$ to $v$ are divisible by $k+1$, and the distances from $u$ to $a$ and from $v$ to $a$ are not equal as $k$ is even. Thus two considered distances have different remainders modulo $k+1$. Hence the geodesic is unique.

\item[(c)] Geodesic between two new vertices, let $u$ and $v$, which lie on the edges $ab$ and $cd$, where $a,b,c,d$ are old vertices (see Figure \ref{fig:dobavlenie_vershin}):

 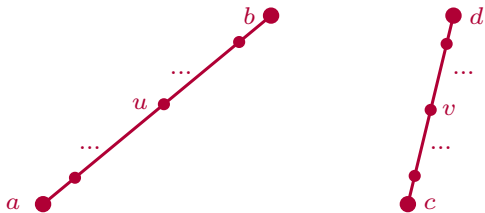
\begin{figure}[htb]
\centering   
\begin{tikzpicture}[color=best, x=0.6cm,y=0.5cm, line width=0.04 cm]
\draw  (0,0)-- (5,5);
\draw  (8,0)-- (9,5);

\draw (1.5,1.5) node[anchor=east] {$...$};
\draw (3.5,3.5) node[anchor=east] {$...$};

\draw (8.25,1.5) node[anchor=west] {$...$};
\draw (8.75,3.5) node[anchor=west] {$...$};

\begin{scriptsize}
\draw [fill=best] (0,0) circle (2.5pt);
\draw [fill=best] (5,5) circle (2.5pt);
\draw [fill=best] (0.7,0.7) circle (1.7pt);
\draw [fill=best] (4.3,4.3) circle (1.7pt);
\draw [fill=best] (2.65,2.65) circle (1.7pt);
\draw[color=best] (-0.25,0) node [anchor=east, scale=1.3]{$a$} ;
\draw[color=best] (2.55,2.65) node [anchor=east, scale=1.3]{$u$} ;
\draw[color=best] (4.9,5) node [anchor=east, scale=1.3]{$b$} ;

\draw [fill=best] (8,0) circle (2.5pt);
\draw [fill=best] (9,5) circle (2.5pt);
\draw [fill=best] (8.85, 4.25) circle (1.7pt);
\draw [fill=best] (8.15, 0.75) circle (1.7pt);
\draw [fill=best] (8.5, 2.5) circle (1.7pt);

\draw[color=best] (8.1,0) node [anchor=west, scale=1.3]{$c$} ;
\draw[color=best] (9.1,5) node [anchor=west, scale=1.3]{$d$} ;
\draw[color=best] (8.5, 2.5) node [anchor=west, scale=1.3]{$v$} ;
\end{scriptsize}
\end{tikzpicture}
\caption{Geodesic between new vertices.}
\label{fig:dobavlenie_vershin}
\end{figure}

Let there be two geodesics between $u$ and $v$. Then it is obvious that both geodesics cannot pass through the vertex $a$, since otherwise this case would contradict to the case \textit{(b)}. Then, without loss of generality, let one of the geodesics pass through $a$ and $c$, and the second through $b$ and $d$. Let's say $d(u,a)=x$, $d(v,c)=y$, then $d(u,b)=k+1-x$, $d(v,d)=k+1-y$. Also note that $k+1 |d(a,c)$ and $k+1|d(b,d)$. Therefore, we assume that $d(u,a)+d(v,c)+d(a,c)=d(u,b)+d(v,d)+d(b,d)$, that is, $d(u,a)+d(v,c) \equiv d(u,b)+d(v,d)\pmod{k+1}$, which is equivalent to $k+1 | 2(x+y)$. Since $k+1$ is odd and $x+y<2(k+1)$, then $x=k+1-y$. But then $d(a,c)=d(b,d)$, that is, from $a$ there are two paths to the vertex $d$ of length $d(a,c)+k+1$ Therefore, since the initial graph $G$ is geodetic, there is a path of the length no more than $d(a,c)$ between the vertices $a$ and $d$, similarly we get that there is a path of length no more than $d(a,c)$ between the vertices $b$ and $c$. But then $d(b,c)+d(c,v)\geq d(b,d)+d(d,v)$ since the geodesic from $u$ to $v$ goes through $b$ and $d$, that is, $d(c,v)\geq d(d,v)$ and similarly, considering the geodesic through the vertex $a$, we get that $d(d,v)\geq d(c,v)$, which cannot be true as $k$ is even.
\end{enumerate}

Necessity. Consider a non-geodetic graph $G$ and prove that $G(k)$ is also non-geodetic:

Let there be two geodesics between the vertices $a$ and $b$ of the graph $G$. Suppose in $G(k)$ between $a$ and $b$ there is a path shorter than the geodesics of $G$ with the added $k$ vertices on each edge. But then there was a shorter path in the initial graph. Hence, in $G(k)$ there are two geodesics between the vertices $a$ and $b$.
\end{proof}

\subsection{Bearing trees and stems}

\begin{proposition} \label{pro: tree vetvi}
Let $T$ be the bearing tree of the geodetic graph $G$, any two stems of which intersect only at the root. And let $T$ contains three stems $s_1$, $s_2$ and $s_3$, such that there are 2 balks between the stems $s_1$, $s_2$ and $s_2$, $s_3$.
Then there is a balk between the stem $s_1$ and the stem $s_3$. Moreover, two balks are drawn in the same tier, and the third is not lower (see Figure \Ref{fig: 2 edges mezdu 3 potocami}).
\end{proposition}
\begin{proof} See Figure \ref{fig: 2 edges mezdu 3 potocami}.
Consider a pair of vertices $w$ and $b$, between them are two paths of length $k-m+2$ (in an even cycle formed by stem $1$ (path from the root to the vertex $b$), stem $3$ (path from the root to the vertex $v$), part of the stem 2 (path from the vertex $u$ to the vertex $c$) and two balks $uv$ and $bc$). And this means that there is a path of the length at least $k-m+1$ between $w$ and $b$, but note that moving from the $m$-th tier to the $k$-th tier we use at least $k-m$ edges, and at least one more edge to get from one stem to another stem. So, this edge must be in one of the tiers from $m$ to $k$. 

\begin{figure}[h]
\centering
\begin{tikzpicture}[scale=0.9,color=best, x=0.6cm,y=0.25cm, line width=0.02cm]
\draw (0,-6) -- (2, -9);
\draw (0,-6) -- (-2, -9) ;
\draw (0, -10.5) -- (0,-9);
\draw (2,-10.5) -- (2, -9);
\draw (-2,-10.5) -- (-2, -9) ;
\draw (0, -13) -- (0,-14.5);
\draw (2,-13) -- (2, -14.5);
\draw (0, -6) -- (0,-9);
\draw[dotted] (-2,-17) arc (180:0:2 and 1);
\draw (0,-17) -- (-2, -17);
\draw (0,-13) -- (2, -13);

\node at (0,-11) {$\dots$};
\node at (0,-15) {$\dots$};
\node[rotate=90] at (-5,-6) {$\dots$};
\node[rotate=90] at (-5,-13) {$\dots$};

\draw [fill=best] (0,-6) circle (2pt);
\draw [fill=best] (0,-9) circle (2pt);
\draw [fill=best] (-2, -9) circle (2pt);
\draw [fill=best] (2, -9) circle (2pt);
\draw [fill=best] (-2,-13) circle (2pt);
\draw [fill=best] (0,-13) circle (2pt);
\draw [fill=best] (2,-13) circle (2pt);
\draw [fill=best] (-2,-17) circle (2pt);
\draw [fill=best] (0,-17) circle (2pt);
\draw [fill=best] (2,-17) circle (2pt);

\begin{scriptsize}
    \node[below][scale=1.25] at (-2, -17) {$b$};
    \node[below] [scale=1.25]  at (0, -17) {$c$};
    \node[below] [scale=1.25]  at (2, -17) {$d$};
    \node[above] [scale=1.25] at (2,-13) {$v$};
    \node[right] [scale=1.25] at (2,-9) {$w$};
    \node[above] [scale=1.25] at (0,-13) {$u$};
    \node[scale=1.25] at (-5, -17) {$k$};
    \node[scale=1.25]  at (-5, -9) {$m$};
    \node[scale=1.25]  at (-2, -5) {$1$};
    \node[scale=1.25] at (0, -5) {$2$};
   \node[scale=1.25] at (2, -5) {$3$};
\end{scriptsize}
\end{tikzpicture}
\caption{Two balks between stems that have no common vertices.}
\label{fig: 2 edges mezdu 3 potocami}
\end{figure}
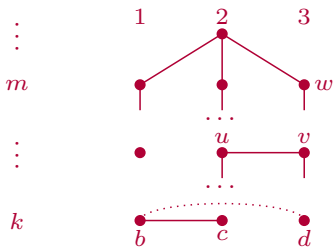

 Further, taking the two <<highest>> edges between the stems, we get that the third edge must be no lower than two others, using Lemma \ref{le: dva rebra mezdu dvumya vetvyami}. If two edges are drawn between the stems in the same tier, then the third edge will not be lower than them.
\end{proof}

\begin{proposition} \label{pro: n peregorodok}
Let $T$ be the bearing tree of the geodetic graph $G$, any two stems of which intersect only at the root. Let $n \ge 2$ and $k$ $\in \N$, and $s_0, s_1, \dots, s_n$ be the stems of the tree $T$, such that there is a balk in the $k$-th tier between the stem $s_0$ and each of the stems $s_1, \dots, s_n$. Then the vertices of the stems $s_1, \dots, s_n$ induce a complete subgraph of the graph $G$ in some of the tiers not lower than the $k$-th.
\end{proposition}
\begin{proof} 
The case where $n=2$ was considered  in Proposition \ref{pro: tree vetvi}. Further, consider the case $n=3$ (see Figure \ref{fig: stebli s n soswdyami}).

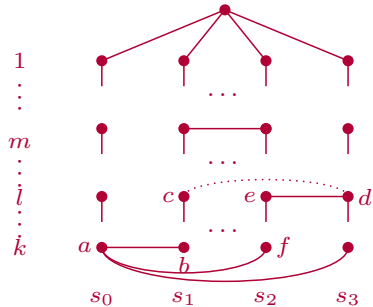
\begin{figure}[htb]
\centering
\begin{tikzpicture}[scale=0.9,color=best, x=0.6cm,y=0.25cm, line width=0.02cm]

\draw (1,-6) -- (2, -9);
\draw (1,-6) -- (-2, -9) ;
\draw (0, -10.5) -- (0,-9);
\draw (2,-10.5) -- (2, -9);
\draw (-2,-10.5) -- (-2, -9) ;
\draw (0, -13) -- (0,-14.5);
\draw (2,-13) -- (2, -14.5);
\draw (-2,-13) -- (-2, -14.5);
\draw (1, -6) -- (0,-9);
\draw[dotted] (0,-17) arc (180:0:2 and 1);
\draw (0,-20) -- (-2, -20);
\draw (1,-6)--(4,-9)--(4,-10.5);
\draw (4,-13)--(4,-14.5);
\draw (4,-17)--(4,-18.5);
\draw (-2,-17)--(-2,-18.5);
\draw (2,-17)--(2,-18.5);
\draw (0,-17)--(0,-18.5);
\draw (-2,-20) arc (-180:0:2 and 1.5);
\draw (-2,-20) arc (-180:0:3 and 2);
\draw(0,-13)--(2,-13);
\draw (2,-17)--(4,-17);

\node at (1,-11) {$\dots$};
\node at (1,-15) {$\dots$};
\node at (1,-19) {$\dots$};
\node[rotate=90] at (-4,-15.5) {$\dots$};
\node[rotate=90] at (-4,-11) {$\dots$};
\node[rotate=90] at (-4,-18.5) {$\dots$};

\draw [fill=best] (1,-6) circle (2pt);
\draw [fill=best] (0,-9) circle (2pt);
\draw [fill=best] (-2, -9) circle (2pt);
\draw [fill=best] (2, -9) circle (2pt);
\draw [fill=best] (-2,-13) circle (2pt);
\draw [fill=best] (0,-13) circle (2pt);
\draw [fill=best] (2,-13) circle (2pt);
\draw [fill=best] (-2,-17) circle (2pt);
\draw [fill=best] (0,-17) circle (2pt);
\draw [fill=best] (2,-13) circle (2pt);
\draw [fill=best] (4,-9) circle (2pt);
\draw [fill=best] (4,-13) circle (2pt);
\draw [fill=best] (4,-17) circle (2pt);
\draw [fill=best] (2,-17) circle (2pt);
\draw [fill=best] (2,-20) circle (2pt);
\draw [fill=best] (-2,-20) circle (2pt);
\draw [fill=best] (0,-20) circle (2pt);
\draw [fill=best] (4,-20) circle (2pt);
\begin{scriptsize}
\node[below][scale=1.25] at (-4, -13) {$m$};
\node[scale=1.25] at (-4, -17) {$l$};
\node[scale=1.25] at (-4, -20) {$k$};
\node[scale=1.25]  at (-4, -9) {$1$};
\node[scale=1.25]  at (-2, -23) {$s_0$};
\node[scale=1.25] at (0, -23) {$s_1$};
\node[scale=1.25] at (2, -23) {$s_2$};
\node[scale=1.25] at (4, -23) {$s_3$};

\node[scale=1.25, anchor=east] at (-2, -20) {$a$};
\node[scale=1.25,anchor=north]  at (0,-20) {$b$};
\node[scale=1.25, anchor=east]  at (0, -17) {$c$};
\node[scale=1.25, anchor=east] at (2, -17) {$e$};
\node[scale=1.25, anchor=west] at (4, -17) {$d$};
\node[scale=1.25,anchor=west] at (2, -20) {$f$};
\end{scriptsize}
\end{tikzpicture}
\qquad
\caption{There is a vertex with $n$ balks going out of it.}
\label{fig: stebli s n soswdyami}
\end{figure}

According to Proposition \ref{pro: tree vetvi}  there are balks in tiers not lower than $k$ between all $s_i$ and $s_j$ $i, j \geq 1$. Suppose two of them lie in different tiers, without loss of generality, let them be the balks between the stems $s_1$ and $s_2$; $s_2$ and $s_3$. From the same Proposition \ref{pro: tree vetvi} it follows that there is an edge between the vertices of the stems $s_1$ and $s_3$; there is a balk in the $l$-th tier (Figure \ref{fig: stebli s n soswdyami}) since the balk between the stems $s_2$ and $s_3$ is lower than the balk between the stems $s_1$ and $s_2$ (we assume this without loss of generality). Then consider the cycle $abcdefa$ whose length is equal to $2(k-l+ 2)$ (see Figure \ref{fig: stebli s n soswdyami}). Note that between the vertices $b$ and $e$ there are two paths of length $k-l + 2$ in this cycle. Also, since the difference between the tiers is $k-l$, there must be a balk between the stems $s_1$ and $s_2$ in one of the tiers between $l$ and $k$. Then we get a contradiction with the fact that the balk between the stems $s_2$ and $s_3$ is lower than the balk between the stems $s_1$ and $s_2$.

Next, we prove by mathematical induction that the statement of the proposition is also true for $n$ balks outgoing from one vertex. The base for three balks is discussed above. Induction step: let the statement be true for $n=p$, then we will prove it for $ n = p + 1 $. Consider triples of stems $s_i, \ s_j, \ s_{p + 1}$ for $1 \leq i, j \leq p$. Then, according to the statement proved above for three arbitrary stems, between $s_i$ and $s_{p + 1}$; $s_j$ and $s_{p + 1}$ there are balks in the same tier as the edge between $s_i$ and $s_j$.
\end{proof}

We denote by $K_n^{\phi}$ a graph that is homeomorphic to a complete graph $K_n$ with nodes $u_r$ $(r =
1,2,…, n)$ together with a function $\phi$ which assigns non-negative integers $\phi(u_r)$ to the nodes such that, given two nodes $u_r$ and $u_s$ of $K_n^{\phi}$, $|d(u_r, u_s)|= \phi(u_r) + 1 + \phi(u_s)$. Where $d(u,v)$ is the distance between vertices $u$ and $v$.

\begin{lemma}[Stemple, Joel G \cite{ex3}] \label{le: sufficient oporn}
A graph G homeomorphic to a complete graph $K_n$ is geodetic if and only if it is a $K_n^{\phi}$
graph.

\end{lemma}

\begin{definition} \label{def:transversal-block}
Let $G$ be a connected graph, $T$ its bearing tree, and $H$ some subgraph of $G$. Suppose that either $H$ is a connected part of some stem of the tree $T$, or all vertices of $H$ lie in the stems $s_0, \dots s_n$ and there exist natural numbers $k$ and $l \le k$ such that the set of all edges $H$ consists of the following:
\begin{itemize}
    \item all the edges of the stems $s_0, \dots s_n$ from the root of the tree to the tier $k$ ($k$ edges in each stem);
    \item one balk from the stem $s_0$ into each of the stems $s_1, \dots, s_n$ in tier $k$;
    \item one balk between any two of the stems $s_1, \dots , s_n$ in the tier $l$.
\end{itemize}
(So, if the subgraph $H$ is not part of one stem, then it has $k(n+1) + 1$ vertices and $k(n+1)+ \frac{n (n + 1)}{2}$ edges.) Then we will say, that the subgraph $H$ \emph{is transversal} to the tree $T$.
\end{definition}

\begin{lemma} \label{le: transversal}
Let the graph $G$ be geodetic, $T$ the bearing tree of the graph $G$, any two stems of which intersect only at the root. And let a subgraph $H$ of the graph $G$ be transversal to $T$, intersects with $n \ge 1$ stems of the graph $T$. And let there be some stem $s_{n + 1}$ that does not intersect with $H$, but from which at least one balk goes into one of the stems intersecting with $H$. Then one can choose a transversal subgraph containing $H$ and intersecting $s_{n + 1}$.
\end{lemma}
\begin{proof}
Let us prove this statement by induction on the number of stems $n$ intersecting $H$. If it intersects with exactly 1 or 2 stems, then Lemma \ref{le: dva rebra mezdu dvumya vetvyami} and Proposition \ref{pro: tree vetvi} come into force. Induction base for $n=3$:

 Suppose that among the first three stems there are two balks in the $k$-th tier and one in the $l$-th (further in the course of the proof, $k$ and $l$ have the same meaning as in the Definition  \ref{def:transversal-block}).
\begin{enumerate}
    \item Let the balk from the fourth stem go to one of the first three stems in the $k$-th tier. If the balk goes to a vertex from which two balks already emerge, then we get the condition for the Proposition \ref{pro: n peregorodok}. Therefore, we will assume that the balk goes to another vertex (see Figure \ref{fig:stebli peregerodra v m yaruse}). Then, according to the Proposition \nolinebreak \ref{pro: tree vetvi}, an edge, marked with a dotted line in the Figure \ref{fig:stebli peregerodra v m yaruse}, should be drawn. And then the subgraph induced by a cycle of length 4 is complete (as the graph is geodetic). So in this case $k=l$.
   
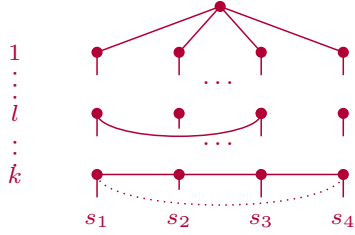
\begin{figure}[h]
\centering
\begin{tikzpicture}[scale=0.9,color=best, x=0.6cm,y=0.225cm, line width=0.02cm]

\draw (1,-6) -- (2, -9);
\draw (1,-6) -- (-2, -9) ;
\draw (0, -10.5) -- (0,-9);
\draw (2,-10.5) -- (2, -9);
\draw (-2,-10.5) -- (-2, -9) ;
\draw (0, -13) -- (0,-14);
\draw (2,-13) -- (2, -14.5);
\draw (-2,-13) -- (-2, -14.5);
\draw (1, -6) -- (0,-9);
\draw (0,-17) -- (-2, -17);
\draw (1,-6)--(4,-9)--(4,-10.5);
\draw (4,-13)--(4,-14.5);
\draw (4,-17)--(4,-18.5);
\draw (-2,-17)--(-2,-18.5);
\draw (2,-17)--(2,-18.5);
\draw (0,-17)--(0,-18);
\draw (-2,-13) arc (-180:0:2 and 1.5);
\draw (-2,-17)[dotted] arc (-180:0:3 and 2);
\draw(0,-17)--(2,-17);
\draw (2,-17)--(4,-17);

\node at (1,-11) {$\dots$};
\node at (1,-15) {$\dots$};
\node[rotate=90] at (-4,-15.5) {$\dots$};
\node[rotate=90] at (-4,-11) {$\dots$};

\draw [fill=best] (1,-6) circle (2pt);
\draw [fill=best] (0,-9) circle (2pt);
\draw [fill=best] (-2, -9) circle (2pt);
\draw [fill=best] (2, -9) circle (2pt);
\draw [fill=best] (-2,-13) circle (2pt);
\draw [fill=best] (0,-13) circle (2pt);
\draw [fill=best] (2,-13) circle (2pt);
\draw [fill=best] (-2,-17) circle (2pt);
\draw [fill=best] (0,-17) circle (2pt);
\draw [fill=best] (2,-13) circle (2pt);
\draw [fill=best] (4,-9) circle (2pt);
\draw [fill=best] (4,-13) circle (2pt);
\draw [fill=best] (4,-17) circle (2pt);
\draw [fill=best] (2,-17) circle (2pt);

\begin{scriptsize}
\node[scale=1.25] at (-4, -13) {$l$};
\node[scale=1.25] at (-4, -17) {$k$};
\node[scale=1.25]  at (-4, -9) {$1$};
\node[scale=1.25]  at (-2, -20) {$s_1$};
\node[scale=1.25] at (0, -20) {$s_2$};
\node[scale=1.25] at (2, -20) {$s_3$};
\node[scale=1.25] at (4, -20) {$s_4$};
\end{scriptsize}
\end{tikzpicture}
\qquad
\caption{The balk from the fourth stem goes in the same tier as the other two.}
\label{fig:stebli peregerodra v m yaruse}
\end{figure}

    \item Let the balk from the fourth stem go to one of the first three stems in some tier with a number not equal to $k$. Then, if the balk goes in the tier above the $k$-th, then according to Proposition \ref{pro: tree vetvi} there is a balk in the $k$-th tier as well,  and this case is already considered. Let the balk from the fourth stem go to some stem in the $a$-th tier, which is lower than $k$. Then,according to the Proposition \nolinebreak \ref{pro: tree vetvi},  there should be balks to all other stems of the given block in the same tier (see Figure \ref{fig:stebli peregerodra nige m}). Then Proposition \ref{pro: n peregorodok} comes into force, which states that $k=l$ in this case.
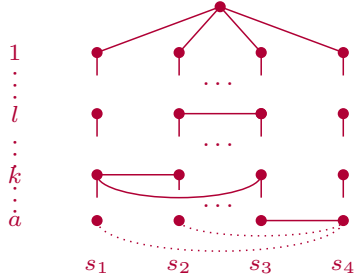
\begin{figure}[h]
\centering
\begin{tikzpicture}[scale=0.9,color=best, x=0.6cm,y=0.225cm, line width=0.02cm]

\draw (1,-6) -- (2, -9);
\draw (1,-6) -- (-2, -9) ;
\draw (0, -10.5) -- (0,-9);
\draw (2,-10.5) -- (2, -9);
\draw (-2,-10.5) -- (-2, -9) ;
\draw (0, -13) -- (0,-14.5);
\draw (2,-13) -- (2, -14.5);
\draw (-2,-13) -- (-2, -14.5);
\draw (1, -6) -- (0,-9);
\draw (0,-17) -- (-2, -17);
\draw (1,-6)--(4,-9)--(4,-10.5);
\draw (4,-13)--(4,-14.5);
\draw (4,-17)--(4,-18.5);
\draw (-2,-17)--(-2,-18.5);
\draw (2,-17)--(2,-18.5);
\draw (0,-17)--(0,-18);
\draw (-2,-17) arc (-180:0:2 and 1.5);
\draw (-2,-20)[dotted] arc (-180:0:3 and 2);
\draw (0,-20)[dotted] arc (-180:0:2 and 1);
\draw(0,-13)--(2,-13);
\draw (2,-20)--(4,-20);

\node at (1,-11) {$\dots$};
\node at (1,-15) {$\dots$};
\node at (1,-19) {$\dots$};
\node[rotate=90] at (-4,-15.5) {$\dots$};
\node[rotate=90] at (-4,-11) {$\dots$};
\node[rotate=90] at (-4,-18.5) {$\dots$};

\draw [fill=best] (1,-6) circle (2pt);
\draw [fill=best] (0,-9) circle (2pt);
\draw [fill=best] (-2, -9) circle (2pt);
\draw [fill=best] (2, -9) circle (2pt);
\draw [fill=best] (-2,-13) circle (2pt);
\draw [fill=best] (0,-13) circle (2pt);
\draw [fill=best] (2,-13) circle (2pt);
\draw [fill=best] (-2,-17) circle (2pt);
\draw [fill=best] (0,-17) circle (2pt);
\draw [fill=best] (2,-13) circle (2pt);
\draw [fill=best] (4,-9) circle (2pt);
\draw [fill=best] (4,-13) circle (2pt);
\draw [fill=best] (4,-17) circle (2pt);
\draw [fill=best] (2,-17) circle (2pt);
\draw [fill=best] (2,-20) circle (2pt);
\draw [fill=best] (-2,-20) circle (2pt);
\draw [fill=best] (0,-20) circle (2pt);
\draw [fill=best] (4,-20) circle (2pt);
\begin{scriptsize}
\node[scale=1.25] at (-4, -13) {$l$};
\node[scale=1.25] at (-4, -17) {$k$};
\node[scale=1.25] at (-4, -20) {$a$};
\node[scale=1.25]  at (-4, -9) {$1$};
\node[scale=1.25]  at (-2, -23) {$s_1$};
\node[scale=1.25] at (0, -23) {$s_2$};
\node[scale=1.25] at (2, -23) {$s_3$};
\node[scale=1.25] at (4, -23) {$s_4$};

\end{scriptsize}
\end{tikzpicture}
\qquad
\caption{The balk from the fourth stem goes below.}
\label{fig:stebli peregerodra nige m}
\end{figure}

\end{enumerate}
We will prove the induction step from $n-1$ stems intersecting $H$ to $n$ as follows:

To begin with, we will prove that if at least one balk from the stem goes to some stem of $H$, then in fact, balks go from this stem into all stems intersecting with $H$. Indeed, suppose we have a transversal subgraph $H$ with $n$ stems $s_1, \dots s_n$ and the stem $s_{n+1}$ from which the balk goes into the stem $s_i$ where $1 \leq i \leq n$. Then, since the subgraph $H$ is transversal, there are balks between any two stems from $H$. Then from $s_i$ there are balks to all the other considered stems, as Proposition \ref{pro: tree vetvi} comes into force.

Further, suppose we have proved the assertion of the lemma for the $n-1$ stems intersecting $H$. Let us prove for $n$ stems: let the stems of $H$ be $s_1, \dots s_n$, and from the stem $s_{n+1}$ there are balks into all these $n$ stems. And let, without loss of generality, in the stem $s_1$ there is a vertex in the lower tier, from which balks go into the stems $s_2, \dots s_n$ (a vertex in the tier $k$).
\begin{enumerate}
    \item If $k>l$, then in the tier $l$ the vertices of the stems $s_2, \dots, s_n$ induce a complete graph, since these stems form $H$. Then consider the stems $s_1, \dots s_{n-1}, s_{n + 1} $. The first $n-1$ stems form a transversal subgraph, and there is a balk from $s_{n+1}$ into this subgraph; therefore, according to the induction hypothesis for the $n-1$ stem, the parts of the stems $s_1, \dots s_{n-1}, s_{n+1}$ also form a transversal subgraph. Then in the tier $l$ the vertices of the stems $s_2, \dots s_{n-1}, s_{n+1}$ induce a complete graph. Indeed, because we have $n-1 \ge 3$ and $k>l$, that is, we have only one option to add to the transversal graph intersecting $n-1$ stems $s_1, \dots s_{n-1}$ one more stem, since the tiers $k$ and $l$ are already fixed, and there can not be two balks between two stems, according to Lemma \ref{le: dva rebra mezdu dvumya vetvyami}. Thus, the vertices of the stems $s_2, \dots s_n$ and the stems $s_2, \dots s_{n-1}, s_{n + 1} $ form complete graphs in the tier $l$, and $n-2 \ge 2$, that is, the vertices of the stems $s_n$ and $s_{n+1}$ of the stems in the tier $l$ must also be adjacent, since at least one cycle on 4 vertices is formed, and since the graph is geodetic this cycle must induce a complete subgraph.
    
    \item If $k=l$, that is, vertices of $n$ stems form a complete graph in the $k$-th tier. Let some balk goes from the stem $s_{n+1}$ in a tier higher than $k$, let it go into the stem $s_i$, then, according to Proposition \ref{pro: tree vetvi}, in the $k$-th tier there are also balks between $s_j$ and $s_{n+1}$, where $1 \le j \le n, \ j \neq i$. Then, since $n \ge 4$, there is a cycle on 4 vertices in the tier $k$ containing the vertices of the stems $s_i$ and $s_{n + 1}$. The subgraph induced by this cycle must be complete. So, we got a contradiction with Lemma \ref{le: dva rebra mezdu dvumya vetvyami}, since there are two balks between the stems $s_i$ and $s_{n+1}$.
    
   Then let there be at least one balk from $s_{n+1}$ to $s_i$ in the $k$-th tier, then according to Proposition \ref{pro: tree vetvi}, there is a balk between $s_{n+1}$ and each of the other stems in one of the tiers not lower than $k$, but there cannot be a balk higher, since this case has already been considered above. That is, in this case, we get a complete graph in the $k$-th tier in stems $s_1, \dots s_{n+1}$.
    
    If there is at least one balk from $s_{n+1}$ to $s_i$ in the tier below the $k$ -th, let it be in the $a$-th tier. Then, according to Proposition \ref{pro: tree vetvi}, there is a balk from $s_{n+1}$ to all other stems in the same $a$-th tier.
\end{enumerate}

In all cases, we obtained either the impossibility of their existence, or their transversality.
\end{proof}

\begin{theorem} \label{th:stebli}
Let the graph $G$ be geodetic, $T$ the bearing tree of the graph $G$, any two stems of which intersect only at the root. A graph $G$ is geodetic if and only if any of its blocks is transversal to $T$.
\end{theorem}
The Figure \ref{fig: example stebli block} shows a geodetic graph with a bearing tree, as in the theorem. It has two large transversal blocks (they contain the root of the tree and are not contained by any stem) and four blocks with one edge (each of them is part of some stem \footnote[1]{If a path is a block, then the length of this path is equal to 1, otherwise it contains a cut vertex. So a transversal block that is contained by some stem consists of two vertices connected by an edge.}).

\begin{proof} 
Necessity. Consider an arbitrary block $B$ of the geodetic graph $G$.

Let us prove its transversality by induction on $n$, where $n$ is the number of stems of the block, parts of which form a transversal subgraph. The base for one stem is obvious. Step: let for $n$ stems intersecting with a block, we already know that parts of these stems form a transversal subgraph. Let us prove that if the block contains some more stems (their parts) other than the given $n$, then we can choose a transversal subgraph such that $n+1$-th stem intersects it. Let the transversal subgraph $H'$ be formed by the parts of the stems $s_1, \dots s_n$. Since more than $n$ stems intersect with a block, there is a stem $s_{n+1}$ that intersects with this block, and from which the balk goes to one of the first $n$ stems. Indeed, after assuming the opposite, we get that the block has a cut vertex (the root of the tree). But then Lemma \ref{le: transversal} comes into force, which states that one can choose a transversal subgraph containing $H'$ and part of the stem $s_{n + 1}$.

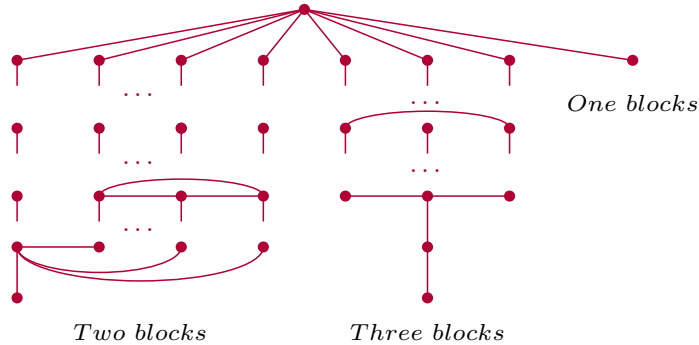
\begin{figure}[H]
\centering
\begin{tikzpicture}[scale=0.9,color=best, x=0.6cm,y=0.25cm, line width=0.02cm]
\draw (1,-6) -- (-2, -9);
\draw (1,-6) -- (-6, -9) ;
\draw (-4, -10.5) -- (-4,-9);
\draw (-2,-10.5) -- (-2, -9);
\draw (-6,-10.5) -- (-6, -9) ;
\draw (-4, -13) -- (-4,-14.5);
\draw (-2,-13) -- (-2, -14.5);
\draw (-6,-13) -- (-6, -14.5);
\draw (1, -6) -- (-4,-9);
\draw (-4,-17) arc (180:0:2 and 1);
\draw (-4,-20) -- (-6, -20);
\draw (1,-6)--(0,-9)--(0,-10.5);
\draw (0,-13)--(0,-14.5);
\draw (0,-17)--(0,-18.5);
\draw (-6,-17)--(-6,-18.5);
\draw (-2,-17)--(-2,-18.5);
\draw (-4,-17)--(-4,-18.5);
\draw (-6,-20) arc (-180:0:2 and 1.5);
\draw (-6,-20) arc (-180:0:3 and 2);
\draw(-4,-17)--(-2,-17);
\draw (-2,-17)--(0,-17);

\draw (1,-6)--(2,-9)--(2,-10.5);
\draw (1,-6)--(4,-9)--(4,-10.5);
\draw (1,-6)--(6,-9)--(6,-10.5);
\draw (2,-13) arc (180:0: 2 and 1);
\draw (2,-13)--(2,-14.5);
\draw (4,-13)--(4,-14.5);
\draw (6,-13)--(6,-14.5);
\draw (2,-17)--(6,-17);
\draw (1,-6)--(9,-9);
\draw (4,-17)--(4,-20);
\draw(-6,-23)--(-6,-20);
\draw (4,-20)--(4,-23);
\node at (-3,-11) {$\dots$};
\node at (-3,-15) {$\dots$};
\node at (-3,-19) {$\dots$};
\node at (4,-11.5) {$\dots$};
\node at (4,-15.5) {$\dots$};

\draw [fill=best] (1,-6) circle (2pt);
\draw [fill=best] (-4,-9) circle (2pt);
\draw [fill=best] (-6, -9) circle (2pt);
\draw [fill=best] (-2, -9) circle (2pt);
\draw [fill=best] (-6,-13) circle (2pt);
\draw [fill=best] (-4,-13) circle (2pt);
\draw [fill=best] (-2,-13) circle (2pt);
\draw [fill=best] (-6,-17) circle (2pt);
\draw [fill=best] (-4,-17) circle (2pt);
\draw [fill=best] (-2,-13) circle (2pt);
\draw [fill=best] (0,-9) circle (2pt);
\draw [fill=best] (0,-13) circle (2pt);
\draw [fill=best] (0,-17) circle (2pt);
\draw [fill=best] (-2,-17) circle (2pt);
\draw [fill=best] (-2,-20) circle (2pt);
\draw [fill=best] (-6,-20) circle (2pt);
\draw [fill=best] (-4,-20) circle (2pt);
\draw [fill=best] (0,-20) circle (2pt);
\draw [fill=best] (-6,-23) circle (2pt);
\draw [fill=best] (2,-9) circle (2pt);
\draw [fill=best] (2,-13) circle (2pt);
\draw [fill=best] (2,-17) circle (2pt);
\draw [fill=best] (4,-9) circle (2pt);
\draw [fill=best] (4,-13) circle (2pt);
\draw [fill=best] (4,-17) circle (2pt);
\draw [fill=best] (6,-9) circle (2pt);
\draw [fill=best] (6,-13) circle (2pt);
\draw [fill=best] (6,-17) circle (2pt);
\draw [fill=best] (4,-20) circle (2pt);
\draw [fill=best] (4,-23) circle (2pt);
\draw [fill=best] (9,-9) circle (2pt);

\begin{scriptsize}
\node[scale=1.25, color=black] at (-3, -25) {$Two \ blocks$};
\node[scale=1.25, color=black] at (4, -25) {$Three \ blocks$};
\node[scale=1.25, color=black] at (9, -11.5) {$One \ blocks$};
\end{scriptsize}
\end{tikzpicture}
    \caption{An example of a geodetic graph with 6 blocks.}
 \label{fig: example stebli block}
\end{figure}

As a result, we get that there is a transversal subgraph $H$ contained in $B$. However, note that there can be no two edges between two stems in a geodetic graph, and that there are no balks between two different blocks, we get that in fact our achieved transversal subgraph $H$ is the block $B$. Indeed, the achieved subgraph $H$ is contained by the same stems that intersect the block $B$, that is, if there are some vertices or edges in the block $B$ other than the vertices and edges in the subgraph $H$, then they are contained in the same stems, but note that there can be no more edges between the stems, since $H$ has edges between any two stems, and there can be no leafs in the block by definition.

Sufficiency. From Lemma \ref{le: sufficient oporn} it follows that every block is geodetic (choosing for the function $\phi$ in for $K_n^{\phi}$ only two vertices, one of which is the root of the bearing tree), and from Lemma \ref{le: block or lobe} it follows that the whole graph $G$ is geodetic.
\end{proof}

\subsection{Hamiltonian graphs} \label{pa: Hamiltonian}



Let $\pi_r$ denote the projective plane of order $r$.
\begin{definition}
The \emph{Levi graph} of the projective plane $\pi_n$, denoted by $Levi(\pi_n) $, is a graph in which the vertices correspond to lines and points of the plane and there is an edge between two vertices if and only if one of the vertices correspond to a point, and the second corresponds to a line containing this point.
\end{definition}
In other words $Levi(\pi_r)$ is the bipartite graph, associated with an incidence structure of the projective plane $\pi_r$.

\begin{lemma} 
[James Singer \cite{projective Hamiltonian for prime}; Lazebnik, Mellinger, Vega \cite{projective Hamiltonian}] \label{le: projective}
The graph $Levi(\pi_r)$ is Hamiltonian.
\end{lemma}

The constructions that we independently found and describe  in Theorem \ref{th: family1} and Theorem \ref{th:dim 2 gam} had also been presented before us in the papers \cite{construction1} and \cite{diam 2}, respectively. However, the authors did not notice that the graphs are Hamiltonian.

\begin{theorem} \label{th: family1} 
For any natural $n$ that is a power of a prime number, there exists a regular geodetic Hamiltonian graph of diameter four with $(n + 1)^3+1$ vertices.
\end{theorem}
\begin{proof}
Let us describe an algorithm for constructing the required graph with the degree of each vertex equal to $n+1$, where $n$ is a power of a prime number.

As well known, there exists a projective plane of order $n$. Consider a bipartite graph $Levi(\pi_{n})$. Let us denote by $P$ one of its parts. Let $u$ be any vertex from the second part and $v_1, \dots, v_{n + 1} \in P$ -- adjacent vertices from the first part. We replace $u$ with the complete subgraph $K_{n+1}$, and instead of edges from $u$ draw  exactly one edge from each vertex of the subgraph into vertices $v_1, \dots, v_{n+1}$ so that no two edges were drawn to the same vertex (in fact, it is a bijection from vertices of $K_{n+1}$ to vertices $v_1, \dots, v_{n+1}$). Let's do this with each vertex from the second part.

Let us prove that the achieved graph $G$ is geodetic, regular, Hamiltonian, and of diameter 4.

Consider an arbitrary vertex $v$ from $P$. It is possible to get from $v$ to any other vertex of $P$ moving along three edges. Indeed, take an arbitrary vertex $v'$ from $P$ and consider the complete subgraph of the graph $G$, with which we have replaced the common neighbor of the vertices $v$ and $v'$ in $Levi(\pi_{n})$. In this complete subgraph, $v$ and $v'$ have exactly one neighbor, and these neighbors are connected by an edge. So there is a path of length three between $v$ and $v'$. Moreover, such a path is unique, since no other vertices adjacent to $v$ and $v'$ from the second part of the graph $G$ are connected by an edge. Hence, between any two vertices of $P$ there is exactly one geodesic (of length 3).

\definecolor{darkgreen}{rgb}{0.0, 0.5, 0.0}  
\begin{figure}[H]
 \centering
 \begin{tikzpicture}[scale=1,color=best, x=0.5cm,y=0.35cm, line width=0.02cm]
\draw (0,0) -- (0,-2)--(2,-4)--(-2,-4)--(0,-2);
\draw (0,0)--(-6,-2)--(-6,-4)--(-8,-4)--(-6,-2);
\draw (0,0)--(6,-2)--(6,-4)--(8,-4)--(6,-2);
\draw (-8,-4)--(-8,-6)--(-8,-8);
\draw (-8,-6)--(-9,-8);
\draw (8,-4)--(8,-6)--(8,-8);
\draw (8,-6)--(9,-8);
\draw(-6,-4)--(-6,-6)--(-6,-8);
\draw (-6,-6)--(-5,-8);
\draw(6,-4)--(6,-6)--(6,-8);
\draw (6,-6)--(5,-8);
\draw (-2,-4)--(-2,-6)--(-3,-8);
\draw (-2,-6)--(-1,-8);
\draw (2,-4)--(2,-6)--(3,-8);
\draw (2,-6)--(1,-8);

\draw[color=blue] (-9,-8) arc (-180:0: 6 and 2);
\draw[color=blue]  (-9,-8) arc (-180:0: 9 and 4);
\draw[color=blue]  (3,-8) arc (-180:0: 3 and 1);

\draw[color=red] (-8,-8) arc (-180:0: 3.5 and 2);
\draw[color=red]  (-8,-8) arc (-180:0: 6.5 and 3);
\draw[color=red]  (-1,-8) arc (-180:0: 3 and 1);

\draw[color=darkgreen] (-6,-8) arc (-180:0: 3.5 and 2);
\draw[color=darkgreen]  (-6,-8) arc (-180:0: 6 and 3);
\draw[color=darkgreen]  (1,-8) arc (-180:0: 2.5 and 1);

\draw[color=black] (-5,-8)--(-3,-8);
\draw[color=black]  (-5,-8) arc (-180:0: 6.5 and 5);
\draw[color=black]  (-3,-8) arc (-180:0: 5.5 and 5);

\draw [fill=best] (0,0) circle (2pt);
\draw [fill=best] (0,-2) circle (2pt);
\draw [fill=best] (6,-2) circle (2pt);
\draw [fill=best] (-6,-2) circle (2pt);
\draw [fill=best] (2,-4) circle (2pt);
\draw [fill=best] (-2,-4) circle (2pt);
\draw [fill=best] (-8,-4) circle (2pt);
\draw [fill=best] (-6,-4) circle (2pt);
\draw [fill=best] (6,-4) circle (2pt);
\draw [fill=best] (8,-4) circle (2pt);
\draw [fill=best] (2,-6) circle (2pt);
\draw [fill=best] (-2,-6) circle (2pt);
\draw [fill=best] (-8,-6) circle (2pt);
\draw [fill=best] (-6,-6) circle (2pt);
\draw [fill=best] (6,-6) circle (2pt);
\draw [fill=best] (8,-6) circle (2pt);
\draw [fill=best] (-1,-8) circle (2pt);
\draw [fill=best] (1,-8) circle (2pt);
\draw [fill=best] (-3,-8) circle (2pt);
\draw [fill=best] (3,-8) circle (2pt);
\draw [fill=best] (-9,-8) circle (2pt);
\draw [fill=best] (-8,-8) circle (2pt);
\draw [fill=best] (-6,-8) circle (2pt);
\draw [fill=best] (-5,-8) circle (2pt);
\draw [fill=best] (6,-8) circle (2pt);
\draw [fill=best] (5,-8) circle (2pt);
\draw [fill=best] (9,-8) circle (2pt);
\draw [fill=best] (8,-8) circle (2pt);
\end{tikzpicture}
 \caption{The case $n=2$.}
\label{fig: example Hamiltonian}
\end{figure}
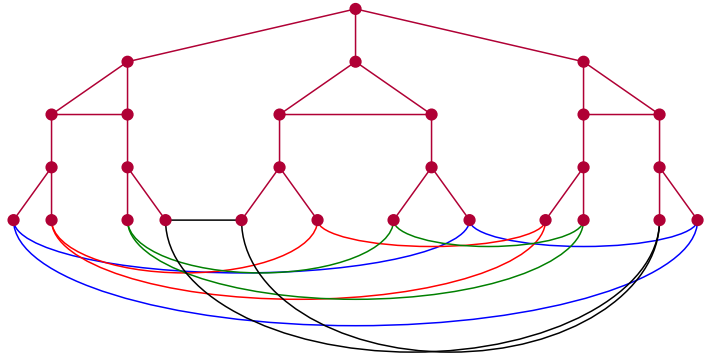

Now take an arbitrary vertex $u$ from the complement $ V(G) \setminus V(P)$. There are two options.

1) If $u$ belongs to the complete subgraph appeared instead of the vertex which was adjacent to $v$ in $Levi(\pi_{n})$, then between $v$ and $u$ there is obviously exactly one geodesic (of length 1 or 2).
   
2) Otherwise,  $u$ belongs to the complete subgraph appeared instead of the vertex which was not adjacent to $v$ in $Levi(\pi_{n})$. The vertex $u$ has exactly one neighbor $v' \neq v$ in the first part $P$, and there is exactly one geodesic (of length \nolinebreak3) between $v'$ and $v$. The path between $u$ and $v$ that doesn't go through $v'$ has the length more than 4.  Therefore, there is also exactly one geodesic (of length 4) between $u$ and $v$. The Figure \nolinebreak\ref{fig: example Hamiltonian} shows an example for $n=2$.

Finally, consider two arbitrary vertices $u, u' \in V(G) \setminus V(P)$ from different complete subgraphs. Since in the graph $ Levi (\pi_{n}) $ any two vertices from one part have a common neighbor from the other part, and also by the construction of the graph $G$, we have three cases:

1) $u$ and $u'$ have a common vertex in $P$;

2) the neighbor $u$ from its complete subgraph and $u'$ has a common vertex in $P$;

3) the neighbor $u$ from its complete subgraph and the neighbor $u'$ from its complete subgraph have a common vertex in $ P $.

In all these cases, between $u$ and $u'$ we get exactly one geodesic of length 2, 3, or 4, respectively.

Thus, the graph $G$ is a regular geodetic and of diameter four. Besides, according to Lemma \ref{le: projective} it's also Hamiltonian. Indeed, we can consider a cycle from Lemma \ref{le: projective} and at the moment when the cycle passes through the vertex of the second part of the graph $Levi(\pi_{n}) $, in the graph of the theorem expand the cycle so that it passes through all the vertices of the complete subgraph. The Figure \ref{fig: example Hamiltonian} shows an example for $n=2$ (Geodetic regular Hamiltonian graph of diameter four on 28 vertices).
\end{proof}

\begin{theorem} \label{th:dim 2 gam}
For any natural $n$ that is a power of a prime number, there exists a geodetic Hamiltonian graph of diameter two with $2n^2+2n+1$ vertices.
\end{theorem}
\begin{proof}
Consider the graph $ Levi(\pi_{n}) $, where $n$ is a power of a prime number. Take some Hamiltonian cycle in this graph (it exists according to the Lemma \ref{le: projective}) and choose an arbitrary pair of adjacent vertices $v$ and $v'$ in the cycle. We replace these vertices with one so that it is adjacent to all neighbors $v$ and $v'$. We denote this vertex by $u$.

Further, we connect all neighbors of the vertex $u$ from the first part of the graph $Levi(\pi_{n})$ with each other and do the same with the vertices of the second part. Then in the bearing tree $T$ of the achieved graph with the root $u$ in the first tier there will be two complete subgraphs on $n$ vertices.

Each vertex of the first tier of $T$ has $n$ different neighbors in the second tier. Then we split the vertices of the first part of the graph $Levi(\pi_{n})$ from the second tier of the tree $T$ into $n$ groups so that in each group any two vertices have a common vertex in the first tier, and connect by an edge all vertices belonging to the same group. So, in the Figure \ref{fig: example Hamiltonian2.0} for $n=2$ there will be two groups of such vertices: $\{1, 2 \}$ and $\{3, 4\}$. The set of all vertices from these $n$ groups, as well as adjacent to them vertices from the first tier, will be denoted by $A$.

Let us prove that the obtained graph $G$ satisfies all the conditions of the theorem.

In the graph $Levi(\pi_{n})$, any two vertices of the same part have exactly one common neighbor. Then obviously between any two non-adjacent vertices from $A$ there is the unique path of length two.  Similarly, there is a unique path of length two  between any vertex from $A$ from the first tier of $T$ and any vertex from the set $B=V(G) - A - \{u \}$ from the second tier $T$ (such path cannot pass through $u$ or through another vertex of the first tier $T$ from $A$). Also obviously there is a unique geodesic from the vertex of the first tier from $A$ to the vertex of the first tier from $B$ (otherwise such vertices would have a common neighbor in the second tier, which would contradict the bipartition of the graph $Levi(\pi_{n})$). The case with geodesics going out from the vertices of the set $B$ of the first tier $T$ is considered analogously.

\definecolor{darkgreen}{rgb}{0.0, 0.5, 0.0}
\definecolor{mygray}{gray}{0.4}
\begin{figure}[H]
\centering
\begin{tikzpicture}[scale=1,color=best, x=0.5cm,y=0.4cm, line width=0.02cm]
\draw (0,0) -- (-2,-2)--(-5,-2)--(0,0);
\draw (0,0) -- (2,-2)--(5,-2)--(0,0);

\draw (-2,-2)--(-1,-4);
\draw (-2,-2)--(-3,-4);
\draw (2,-2)--(1,-4);
\draw (2,-2)--(3,-4);

\draw (-5,-2)--(-5,-4);
\draw (-5,-2)--(-7,-4);
\draw (5,-2)--(5,-4);
\draw (5,-2)--(7,-4);
\draw (-5,-4)--(-7,-4);
\draw (-1,-4)--(-3,-4);

\draw[color=blue] (-7,-4) arc (-180:0: 4 and 2);
\draw[color=blue]  (-5,-4) arc (-180:0: 4 and 2);
\draw[color=red] (-7,-4) arc (-180:0: 6 and 3);
\draw[color=red]  (-5,-4) arc (-180:0: 6 and 3);

\draw[color=darkgreen] (-3,-4) arc (-180:0: 2 and 1);
\draw[color=darkgreen]  (-1,-4) arc (-180:0: 2 and 1);

\draw[color=black] (-3,-4) arc (-180:0: 5 and 2.5);
\draw[color=black]  (-1,-4) arc (-180:0: 3 and 1.5);

\draw [fill=best] (0,0) circle (2pt);
\draw [fill=best] (-2,-2) circle (2pt);
\draw [fill=best] (-5,-2) circle (2pt);
\draw [fill=best] (2,-2) circle (2pt);
\draw [fill=best] (5,-2) circle (2pt);
\draw [fill=best] (-5,-4) circle (2pt);
\draw [fill=best] (-7,-4) circle (2pt);
\draw [fill=best] (-1,-4) circle (2pt);
\draw [fill=best] (-3,-4) circle (2pt);
\draw [fill=best] (1,-4) circle (2pt);
\draw [fill=best] (3,-4) circle (2pt);
\draw [fill=best] (5,-4) circle (2pt);
\draw [fill=best] (7,-4) circle (2pt);

\draw [dotted] [color=mygray] (-4,-3) ellipse (1.95cm and 1cm);
\draw [dotted] [color=mygray] (4,-3) ellipse (1.95cm and 1cm);
\begin{scriptsize}
\node[color=black, anchor=south][scale=1.3] at (0,0) {$u$};
\node[color=black, anchor=east][scale=1.3] at (-7,-4) {$1$};
\node[color=black, anchor=west][scale=1.3] at (-5,-4) {$2$};
\node[color=black, anchor=east][scale=1.3] at (-3,-4) {$3$};
\node[color=black, anchor=west][scale=1.3] at (-1,-4) {$4$};

\node[color=black, anchor=west][scale=1.3] at (-8,-1) {$A$};
\node[color=black, anchor=east][scale=1.3] at (8,-1) {$B$};
\end{scriptsize}
\end{tikzpicture}
\caption{Geodetic Hamiltonian graph of diameter two on 13 vertices (the case $n=2$).}
\label{fig: example Hamiltonian2.0}
\end{figure}
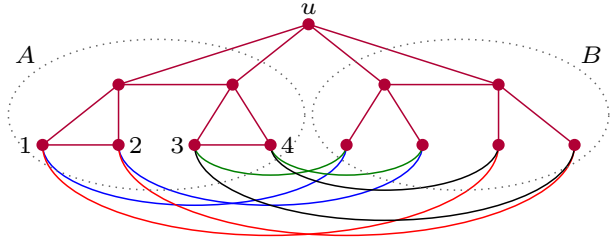

Consider now two vertices $x$ and $y$ of the graph $G$ from the second tier $T$.

1) If they both belong to the same set $A$ or $B$, then there is a path of length one or two between them, since in the graph $Levi(\pi_{n})$ they correspond to the vertices of the same part. The uniqueness of such geodesic also follows from the property of the graph $Levi(\pi_{n})$.

2) Let $x \in A$ and $y \in B$. Consider the group which contains $x$ in $A$ (see above). By construction, any vertex from the set $B$ of the second tier $T$ is adjacent to exactly one vertex of this group (since it has exactly one common neighbour with each vertex from $A$ of the first tier of $T$ in the $Levi(\pi_{n})$). Hence, $x$ and $y$ are adjacent or there is exactly one path of length two between them.

Thus, the graph $G$ is geodetic and of diameter 2. And since it is obtained from the Hamiltonian graph $Levi(\pi_{n})$ by replacing two consequent vertices in an Hamiltonian cycle with a vertex adjacent to all their neighbors and adding several edges, the graph $G$ is also Hamiltonian.
\end{proof}

\subsection{Subclasses of Hamiltonian graphs}

\begin{definition}

A graph is called \emph{locally connected} if an open neighborhood\footnote[1]{a subgraph induced by vertices that are adjacent to a given one} of any vertex of this graph induces a connected subgraph.
\end{definition}

\begin{proposition}
Let $G$ be a geodetic locally connected graph. Then $G$ is a complete graph.
\end{proposition}
\begin{proof}

Consider a bearing tree, fixing an arbitrary vertex as the root. According to the lemma \ref{le: K_n v pervom yaruse}, all vertices of the first tier of this tree are divided into several complete subgraphs. Therefore, due to local connectivity, there is always exactly one complete subgraph in the first tier. So the whole graph $G$ is complete.
\end{proof}

\begin{definition}
A cycle $C$ in a graph $G$
is \emph{extendable} if there is a cycle $C'$ in $G$ that contains all the vertices of C and exactly one additional vertex.  A graph $G$ is \emph{cycle extendable} if every non-Hamiltonian cycle in
G is extendable. If in addition every vertex of G belongs to a 3-cycle, then $G$ is \emph{fully
cycle-extendable}.

\end{definition}

\begin{lemma} \label{le: fully extandable}
Let $G$ be a geodetic fully cycle extendable graph. Then $G$ is a complete graph.
\end{lemma}
\begin{proof}

Let the graph $G$ is not complete, and in the maximal by the number of vertices clique $K_m$ of the graph $G$ there are $m$ vertices. Then, since this clique is a complete subgraph containing at least three vertices, and the graph $G$ is cycle extendable, then there is a cycle on $m+1$ vertex containing this clique $K_m$. If $v$ is a vertex from this cycle that is not contained in $K_m$, then $v$ is adjacent to all vertices of $K_m$. Indeed, otherwise, if the vertex $u \in K_m$ is not adjacent to $v$, then there would be two different geodesics of length two between $u$ and $v$ (there are at least two edges from $v$ to the complete graph $K_m$). Thus we got a contradiction, since we got a click on $m+1$ vertices.
\end{proof}

\begin{proposition}
Let $G$ be a geodetic cycle expandable graph. Then $G$ is a complete graph or a tree.
\end{proposition}
\begin{proof}

It follows from the proof of the Lemma \ref{le: fully extandable} that if there is a triangle in a geodetic cycle extendable graph, then the graph is complete. We show that the graph $G$ has a triangle if it is not a tree. Consider a cycle of the smallest length, let it be a cycle $C$ of length $k$. Then consider a cycle $C'$ of length $k+1$ containing all vertices of the cycle $C$. Some edges of the cycle $C$ are edges that do not belong to the the edges of cycle $C'$ . Consider one of such edges, it divides the cycle $C'$ into two smaller ones. And then either $k=3$, or the length of one of the cycles is less than $k$, which contradicts the assumption that $C$ is the cycle of the smallest length.
\end{proof}

\subsection{Hereditary classes for geodetic graphs} \label{pa: heriditary}
\noindent
\begin{definition}
A class $X$ of graphs  is called \emph{hereditary} if for each graph $G \in X$ all induced subgraphs of $G$ also belong to $X$. In other words, $X$ is hereditary if and only if
$G \in X$ implies $G - v \in X$ for any vertex $v$ of $G$, where $G - v$ denotes the graph
obtained from $G$ by removing $v$ together with all edges incident to it.
\end{definition}


\begin{definition}
Given a class of graphs $M$ (maybe infinite), we denote by $Free(M)$ the class of graphs containing no graph from $M$ as an induced subgraph.
We say that graphs in $M$ are \emph{forbidden induced subgraphs} for the class $Free(M)$ and
that graphs in $Free(M)$ are \emph{$M$-free}.
\end{definition}

\begin{lemma} [Kitaev and Lozin \cite{hereditary}] \label{le: free}
A class $X$ of graphs is hereditary if and only if there is a set $M$ such that $X = Free(M)$.
\end{lemma}

\begin{definition}
A graph $G$ is a minimal forbidden induced subgraph for a hereditary class $X$ if $G$ does not belong to $X$ but every induced subgraph of $G$ not equal to $G$ belongs to $X$ (or alternatively, the deletion of any vertex from $G$ results in a graph that belongs to $X$).
\end{definition}
\noindent Let us denote by $Forb(X)$ the set of all minimal forbidden induced subgraphs for a hereditary
class $X$.

\begin{lemma}[Kitaev and Lozin \cite{hereditary}] \label{le: forb}
For any hereditary class $X$, we have $X = Free(F orb(X))$. 
\end{lemma}

\begin{definition}
Let $P$ be an arbitrary class of graphs (not necessarily hereditary). We denote by
$\lfloor P \rfloor$ the maximum hereditary subclass of the class $P$, and by $\lceil P \rceil$ the minimal superclass of class $P$. 
\end{definition}
 
 \begin{lemma} \label{le: cdsjco}
 $\lfloor P \rfloor$ и $\lceil P \rceil$ are uniquely defined.
 \end{lemma}
 
 \begin{proof} 
 Let $\lfloor P_1 \rfloor$ and $\lfloor P_2 \rfloor$ be two  hereditary subclasses of class $P$, and neither of them contains the other. Then, considering the union of $\lfloor P_1 \rfloor$ and $\lfloor P_2 \rfloor$, we obtain a subclass of the class $P$, the heredity of which obviously follows from the fact that $\lfloor P_1 \rfloor$ and $\lfloor P_2 \rfloor$ are also hereditary classes. So we achieved that one subclass contains the others.
 
  Let $\lceil P_1 \rceil$ and $\lceil P_2 \rceil$ be two hereditary superclasses of class $P$, and neither of them contains the other. Then, considering the intersections of these classes, we get a hereditary superclass of the class $P$. Indeed, the heredity of this superclass will follow from the following property of hereditary classes: if $\lceil P_1 \rceil = Free(Z_1)$ and $\lceil P_2 \rceil = Free (Z_2)$, then $\lceil P_1 \rceil \cap \lceil P_2 \rceil = Free (Z_1 \cup Z_2)$. This property has been proved in the article \cite{hereditary}.
 \end{proof}

\noindent Further we consider $\mathcal{G}$ -- the class of geodetic graphs. We will assume that every connected component of an induced subgraph of a graph from the class $\lfloor \mathcal{G} \rfloor$ must be geodetic. So we may consider only connected induced subgraphs of graphs from the class $\lfloor \mathcal{G} \rfloor$.

We denote by $\mathcal{C}'_{2k}$ the set of graphs obtained from a cycle on $2k$ vertices by adding one edge dividing the cycle into two odd cycles. The example is shown in the Figure \ref{fig: c_2t'}.

\begin{figure}[h]
\centering
\begin{tikzpicture}[color=best, x=0.5cm,y=0.3cm, line width=0.02cm]
\draw (6,0) -- (7,2)--(9,2);
\draw (6,0)--(7,-2)--(9,-2);
\draw (12,0) -- (11,2)--(9,2);
\draw (12,0) -- (11,-2)--(9,-2);
\draw (11,2)--(11,-2);
\draw [fill=best] (6,0) circle (2pt);
\draw [fill=best] (7,-2) circle (2pt);
\draw [fill=best] (7,2) circle (2pt);
\draw [fill=best] (9,-2) circle (2pt);
\draw [fill=best] (9,2) circle (2pt);
\draw [fill=best] (11,-2) circle (2pt);
\draw [fill=best] (11,2) circle (2pt);
\draw [fill=best] (12,0) circle (2pt);
\end{tikzpicture}
\hspace{2cm}
\begin{tikzpicture}[color=best, x=0.5cm,y=0.3cm, line width=0.02cm]
\draw (6,0) -- (7,2)--(9,2);
\draw (6,0)--(7,-2)--(9,-2);
\draw (12,0) -- (11,2)--(9,2);
\draw (12,0) -- (11,-2)--(9,-2);
\draw (9,2)--(9,-2);
\draw [fill=best] (6,0) circle (2pt);
\draw [fill=best] (7,-2) circle (2pt);
\draw [fill=best] (7,2) circle (2pt);
\draw [fill=best] (9,-2) circle (2pt);
\draw [fill=best] (9,2) circle (2pt);
\draw [fill=best] (11,-2) circle (2pt);
\draw [fill=best] (11,2) circle (2pt);
\draw [fill=best] (12,0) circle (2pt);
\end{tikzpicture}
\caption{The set $\mathcal{C}'_{8}$ consists of two graphs.}
\label{fig: c_2t'}
\end{figure}
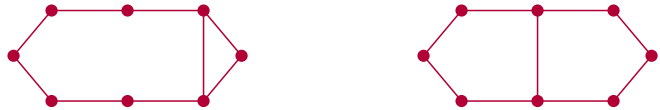

\begin{lemma} \label{le: FIS(P_)}
 $\lfloor \mathcal{G} \rfloor =Free(C_{2k}, \mathcal{C}'_{2k} \ | k \geq \nolinebreak 2)$.
\end{lemma}

\begin{proof}
Find all subgraphs of $H$ that are not geodetic. Then the graph $H$ contains two vertices $u$ and $v$, between which there are at least two geodesics.

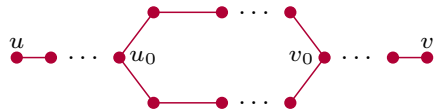
\begin{figure}[H]
\centering
\begin{tikzpicture}[color=best, x=0.45cm,y=0.3cm, line width=0.02cm]
\draw (3,0) -- (4,0);
\draw (6,0) -- (7,2)--(9,2);
\draw (6,0)--(7,-2)--(9,-2);
\draw (15,0) -- (14,0);
\draw (12,0) -- (11,2);
\draw (12,0) -- (11,-2);
\draw [fill=best] (3,0) circle (2pt);
\draw [fill=best] (4,0) circle (2pt);
\node[color=black] at (5,0) {$\dots$};
\draw [fill=best] (6,0) circle (2pt);
\draw [fill=best] (7,-2) circle (2pt);
\draw [fill=best] (7,2) circle (2pt);
\draw [fill=best] (9,-2) circle (2pt);
\draw [fill=best] (9,2) circle (2pt);
\node[color=black] at (10,2) {$\dots$};
\node[color=black] at (10,-2) {$\dots$};
\draw [fill=best] (11,-2) circle (2pt);
\draw [fill=best] (11,2) circle (2pt);
\draw [fill=best] (12,0) circle (2pt);
\draw [fill=best] (14,0) circle (2pt);
\node[color=black] at (13,0) {$\dots$};
\draw [fill=best] (15,0) circle (2pt);

\begin{scriptsize}
\node[color=black, anchor=south][scale=1.25] at (3,0) {$u$};
\node[color=black, anchor=south][scale=1.25] at (15,0) {$v$};
\node[color=black, anchor=east][scale=1.25] at (12,0) {$v_0$};
\node[color=black, anchor=west][scale=1.25] at (6,0) {$u_0$};
\end{scriptsize}
\end{tikzpicture}
\caption{A forbidden subgraph of $\lfloor \mathcal{G} \rfloor$.}
\label{fig: FIS(P_)}
\end{figure}
Consider some two geodesics between two vertices $u$ and $v$. Let the first time they diverge at the vertex $u_0$ (the first from the vertex $u$), and the first time they converge at the vertex $v_0$. Then consider a cycle whose opposite ends are $u_0$ and $v_0$. If there are some edges in this cycle, then they are "vertical", that is, "perpendicular" to $u_0v_0$ (see Figure \nolinebreak\ref{fig: FIS(P_)}). Indeed, considering other edges we obtain that the path $u_0v_0$ of the cycle has not the smallest distance between the vertices. Let at least two "vertical" edges be drawn in this cycle. Then we can select a smaller cycle (according to Lemma \ref{le: forb}, it is enough to consider minimal forbidden subgraphs). Therefore, all forbidden subgraphs - even cycles and even cycles with an edge and there are no other minimal forbidden ones.

\end{proof}

\begin{theorem} \label{th: FIS(P_)}
Let $G \in \lfloor \mathcal{G} \rfloor$. Then each block of the graph $G$ is a complete subgraph or an odd-length cycle.
\end{theorem}
\begin{proof}
Consider an arbitrary connected graph $G$ from the class $\lfloor \mathcal{G} \rfloor$. Choose the smallest even cycle in $G$ that does not induce a complete subgraph. It must have at least two diagonals (edges whose ends are the vertices of the cycle that don't belong to the edges of the cycle), according to Lemma \ref{le: FIS(P_)}. Then we can select an even cycle with the length less than the length of the initial cycle - either one of the diagonals divides the cycle into two even ones, or we can select an even <<crossing>> cycle formed by two diagonals, unless of course this is $K_4$ - a complete graph on 4 vertices, but we do not consider it from the very beginning. Then this smaller cycle must be a complete graph. And then there is a diagonal, let $uv$, dividing the original even cycle into two smaller even ones. But then each of these even cycles is a complete graph. And then each vertex of the cycle will be adjacent to both vertices $u$ and $v$, that is, we can select cycles of length 4, the induces complete subgraphs. And then the entire induced subgraph of this cycle is a complete graph.

This way we can choose the smallest even cycle that does not induce a complete subgraph in each block. That is, we got that either each block is an odd cycle, or it is a complete subgraph.

Also note that the considered above reasoning is a necessary condition. Sufficiency can be easily checked.
\end{proof}

\begin{lemma}
Let $K_4-e$ be a graph consisting of two triangles with one common side. Then $Free(C_4, K_4 -e)$ is a hereditary class of graphs containing a class geodetic graphs.
\end{lemma}
\begin{proof}
Let some geodetic graph doesn't belong to the class $Free(C_4, K_4 -e)$. Then this geodetic graph must contain $C_4$ or $K_4 -e$ as an induced subgraph.
\end{proof}

\begin{remark}
To prove that $\lceil \mathcal{G} \rceil = Free(C_4, K_4 -e)$ it is necessary to show that any graph $H$, not containing $C_4$ or $K_4-e$ as an induced subgraph, is an induced subgraph of some geodetic graph.
\end{remark}

\section{Antipodal graphs}
\begin{proposition} \label{pro: monoant tree}
A tree is an antipodal graph if and only if the longest path of a tree is unique and has an odd length.
\end{proposition}
\begin{proof}
Necessity. Let the tree be an antipodal graph and assume it has two longest simple paths. Then it is clear that these paths intersect. Indeed, since if this were not true, their lengths could have been increased. Indeed, considering these paths and a path $l$ connecting them we can take at least half of each of the longest path (depending on the points of intersection of $l$ and longest paths) and the path $l$ itself, receiving a longer path.

 Let, without loss of generality, the longest paths be $ab$ and $cd$. And they intersect at the  path $uv$ (see Figure \nolinebreak\ref{fig:dva dlinnyh puti}). Then $d(a, v) \geq d(c,v) $, because otherwise $ab$ could be chosen longer (the path $bc$). Similarly, we obtain that $d(c, v) \geq d(a, v)$, that is, $d(a, v)=d(c, v)$, similarly we obtain the equality $d(u,d)=d(u,b)$. Let, without loss of generality, $d(a,v) \leq d(u, b)$, but then the vertex $a$ has at least two antipodes - $b$, $d$.

Thus, we have proved that if a tree is an antipodal graph, then it has exactly one longest path, but then if it is unique, then it is obviously odd. Indeed, if it were even, then there would be the middle of the given path, equidistant from the ends of the path. Therefore, there would be some vertex even further from the middle than the ends of the longest paths. This would mean that the initially chosen path is not the longest.

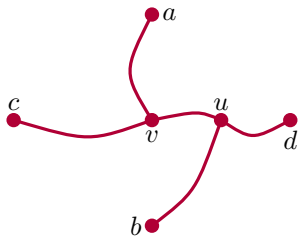
\begin{figure}[H]
\centering
\begin{tikzpicture}[scale=0.7,x=1.3cm, y=1cm, color=best,line width=0.04cm,
block/.style={circle,fill=best,draw=best,inner sep=1.5pt}]
    \foreach \i in {1,...,4}
    {
        \coordinate (\i) at (\i*360/4:2);
        \node[block] at (\i*360/4:2) {};
    }
 \coordinate (15) at (0:0);
        \node[block] at (0:0) {};
\coordinate (5) at (0:1);
\node[block] at (0:1) {};

    \node[color=black,anchor=west] at (1) {$a$};
    \node[color=black,above] at (2) {$c$};
    \node[color=black,left] at (3) {$b$};
    \node[color=black,below] at (4) {$d$};
    \node[color=black,anchor=north] at (0:0) {$v$};
    \node[color=black,anchor=south] at (0:1) {$u$};
    \begin{scope}[color=best]
    \draw[very thick] (90:2) .. controls (115:1) .. (0:0);
    \draw[very thick]  (0:0).. controls (15:0.7) ..(0:1);
    \draw[very thick]  (0:1).. controls (295:1.5) ..(270:2);
    \draw[very thick] (180:2) .. controls (205:1) .. (0:0);
    \draw[very thick]  (0:1).. controls (-15:1.5) ..(0:2);
    \end{scope}
\end{tikzpicture}
\caption{Two longest simple paths $ab$ and $cd$ intersect at the path $uv$.}
\label{fig:dva dlinnyh puti}
\end{figure}

 Sufficiency. Let the longest simple path of a tree be unique and odd (denote this path by  $P$ with endpoints $a$ and $b$). It is clear that each vertex of a given path has exactly one antipode. Therefore, further we will consider the other vertices. Suppose some vertex $c$ has an antipode $d$ which is neither of the ends of $P$. Let the closest to $c$ vertex from the path $P$ be the vertex $u$. Then $d(u,d) \geq d(u,a)$, otherwise, the path $cd$ would be shorter than the path $ca$ as the following is true $d(c,a)=d(c,u)+d(u,a) \leq d(c,d) \leq d(c,u)+d(u,d)$. But from the inequality $d(u,d) \geq d(u,a)$ follows that the longest path $P$ is not unique. Also it's obvious that there is no vertex that can be equidistant from $a$ and $b$, since the longest path is odd.
\end{proof}

\begin{proposition} \label{pro: any graph in antipodal}
It's possible to add few vertices and edges to any graph $G$ so that it becomes antipodal.
\end{proposition}

\begin{proof}
Let the diameter of the graph $G$ be equal to $d$. Then we add a simple path of length $2d+1$ passing through exactly one arbitrary vertex of the graph $G$. Moreover, this vertex is one of the midpoints of this simple path. In the Figure \ref{fig:make monoant graph} the considered vertex is $v_0$. It is clear that the antipode of any vertex of the graph $G$ is $v_{-(d+1)}$. Now consider the antipodes of $v_i$. The antipode of $v_0$ is $v_{-(d+1)}$, since in the graph $G$ the maximum distance to $v_0$ is equal to $d$. Now consider the vertex $v_i$, where $i \in \N$. On the added path, its antipode is $v_{-(d+1)}$, which is at a distance of $i+d+1$ from it. Moreover, any vertex of the graph $G$ is at the distance no more than $i+d$ from $v_i$  (the distance to $v_0$, and the maximum distance from $v_0$ to the vertices of the graph $G$). Thus, $v_i$ has exactly one antipode for $i \geq 0$. 

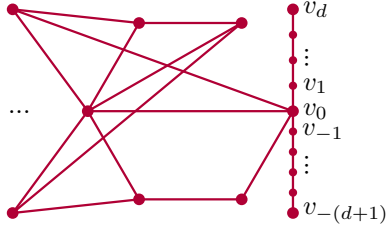
\begin{figure}[H]
\centering
\begin{tikzpicture}[scale=0.9,color=best,  line width=0.03cm,
block/.style={circle,fill=best,draw=best,inner sep=1.25pt}]
    \foreach \i in {1,...,6}
    {
        \coordinate (\i) at (\i*360/6:1.5);
        \node[block] at (\i*360/6:1.5) {};
    }
    \coordinate (7) at (150:3);
        \node[block] at (150:3){};
     \coordinate (8) at (-150:3);
        \node[block] at (-150:3){};
        \node[color=black] at (180:2.5) {$...$};
        \node[color=black, anchor= west] at (0:1.5) {$v_0$};

         \coordinate (9) at (45:2.1211320);
        \node[block] at (45:2.1211320){};
                 \coordinate (10) at (-45:2.1211320);
        \node[block] at (-45:2.1211320){};
        
        \foreach \x in {0.2,0.4,0.6,0.8}
        \filldraw ($(10)!\x!(6)$) circle (1.25pt);
        
                   \node[color=black, anchor=west] at (-12:1.52069063257) {$v_{-1}$};
                      \node[color=black, anchor=west] at (-45:2.1211320) {$v_{-(d+1)}$};
        \node[color=black,rotate=90] at (-25:1.87705098312) {$...$};
        \foreach \x in {0.25,0.5,0.75}
        {
        \filldraw ($(9)!\x!(6)$) circle (1.25pt);
        }
           \node[color=black, anchor=west] at (14:1.5461646096) {$v_1$};
                      \node[color=black, anchor=west] at (45:2.1211320) {$v_{d}$};
        \node[color=black,rotate=90] at (25:1.87705098312) {$...$};

    \begin{scope}[color=best]
\draw (1) --(2);
\draw (3) --(2);
\draw (4) --(3);
\draw (5) --(4);
\draw (5) --(6);
\draw (3) --(6);
\draw (1) --(3);
\draw (7) --(3);
\draw (7) --(2);
\draw (7) --(6);
\draw (3) --(8);
\draw (4) --(8);
\draw (1) --(8);
\draw (10) --(9);
    \end{scope}
\end{tikzpicture}
\caption{Converting an arbitrary graph to an antipodal graph.}
\label{fig:make monoant graph}
\end{figure}

The case $v_{- i}$ is checked analogously to the case $v_{i}$, where $i$ is positive.
\end{proof}



\noindent
Let $\mathcal{A}$ be the class of antipodal graphs. We will assume that every connected component of an induced subgraph of a graph from the class $\lfloor \mathcal{A} \rfloor$ is antipodal. So, we may consider only connected induced subgraphs of graphs from the class $\lfloor \mathcal{A} \rfloor$.
\begin{theorem} \label{th: ant her}
Subclass $\lfloor \mathcal{A} \rfloor = \{P_1, P_2\}$, where $P_1, P_2$ are paths on one and two vertices, respectively. The superclass $\lceil \mathcal{A} \rceil$ contains all connected graphs.
\end{theorem}
\begin{proof}
Let us prove that no graph $G \in \lfloor \mathcal {A} \rfloor $ contains a triangle. Suppose the opposite, then we choose this triangle as the induced subgraph of the graph $G$, by definition it must be antipodal, but each vertex has 2 antipodes.

Besides, graph $G \in \lfloor \mathcal{A} \rfloor$ does not contain a simple path of length two as an induced subgraph. Assuming the opposite, consider this path as an induced subgraph, it must be an antipodal graph by the definition, but the central vertex of the path has two antipodes.

Now suppose that some connected subgraph doesn't belong to $\lceil \mathcal{A} \rceil$. Then, according to Lemma \ref{le: free}, there must exist some nonempty forbidden induced subgraph, let it be a graph $H$ that is not contained in any graph $G \in \lceil \mathcal{A} \rceil$. However, from the proof of Proposition \ref{pro: any graph in antipodal}, there is an antipodal graph containing $H$ as an induced subgraph.
\end{proof}

\begin{proposition} \label{pro: 2 soseda 2 svyazniy}
Let $G$ be a geodetic graph and the degree of each vertex is at least two. Then any of its vertices of the graph $G$ has at least two antipodes.
\end{proposition} 
\begin{proof}
Let's prove the statement assuming the opposite. Let some vertex $a$ has exactly one antipode $b$. Then let's build a bearing tree with the root vertex $a$. Then in the last, let's say the $k$th tier (Figure \ref{fig: opornoe s dobavleniem}), there is exactly one vertex, since vertex $a$ has exactly one antipode. Let's add edges that are not in the bearing tree, but are in the graph $G$. Then consider vertex $b$ and the edges coming out of it. Edges cannot go to the tier with the number $k+1$, since $b$ is the antipode of the vertex $a$, also only one edge can go to the $k-1$-th tier, according to the Lemma \ref{le:opornoe}. Thus, there is at least one more vertex in the $k$-th tier, connected with vertex $b$.

\begin{figure}[htb]
\centering
\begin{tikzpicture}[color=best, x=0.6cm,y=0.35cm, line width=0.02cm]
\draw  (0,0)-- (-2,-2);
\draw  (0,0)-- (2, -2);
\draw  (0,0)-- (0, -2);
\draw  (0,0)-- (-4, -2);
\draw  (0,0)-- (4, -2);
\draw (0,-12) -- (0, -9);
\draw (0,-6) -- (0, -9);
\draw (3, -6) -- (2, -9);
\draw (3, -6) -- (4, -9);
\draw (-2, -2) -- (-4,-2);
\draw (2, -9) -- (4, -9);
\draw (0, -6) -- (-2, -9);
\draw (3, -6) -- (6, -9);
\draw (3, -6) -- (8, -9);

\node[rotate=90] at (0,-4) {$\dots$};
\node[rotate=90] at (-5,-4) {$\dots$};

\draw [fill=best] (0,0) circle (2pt);
\draw [fill=best] (-2,-2) circle (2pt);
\draw [fill=best] (2,-2) circle (2pt);
\draw [fill=best] (-4, -2) circle (2pt);
\draw [fill=best] (4, -2) circle (2pt);
\draw [fill=best] (0, -2) circle (2pt);
\draw [fill=best] (0,-6) circle (2pt);
\draw [fill=best] (0,-9) circle (2pt);
\draw [fill=best] (0,-12) circle (2pt);
\draw [fill=best] (3,-6) circle (2pt);
\draw [fill=best] (2,-9) circle (2pt);
\draw [fill=best] (4,-9) circle (2pt);
\draw [fill=best] (-2,-9) circle (2pt);
\draw [fill=best] (6,-9) circle (2pt);
\draw [fill=best] (8,-9) circle (2pt);
\begin{scriptsize}
    \node[above][scale=1.25]  at (0,0.1) {$a$};
    \node[below][scale=1.25] at (0, -12) {$b$};
    \node[scale=1.25]  at (-5, -9) {$k-1$};
    \node[scale=1.25] at (-5, -12) {$k$};
    \node[scale=1.25]  at (-5, -6) {$k-2$};
    \node[scale=1.25]  at (-5, 0) {$0$};
    \node[scale=1.25] at (-5, -2) {$1$};
\end{scriptsize}
\end{tikzpicture}
\caption{A bearing tree with the root $a$ and exactly one antipode.}
\label{fig: opornoe s dobavleniem}
\end{figure}
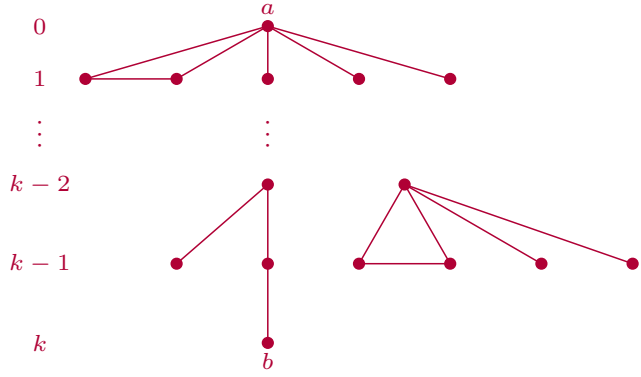
\end{proof}

\begin{corollary}
Any vertex of a geodetic Hamiltonian graph has at least two antipodes.
\end{corollary}

\section{Weighted graphs}
\begin{definition}
The \emph{length} of a path in a weighted graph is the sum of the weights of its edges.
A \emph{geodesic} between two vertices of a weighted graph is the path between the given vertices with the smallest length. 
A weighted graph is called \emph{geodetic} if there is exactly one geodesic between any two of its vertices.
\end{definition}

\begin{theorem} \label{th: vzveshen}
It's possible to assign weights to edges of any connected graph $G$ so that it becomes geodetic and antipodal simultaneously.
\end{theorem}
\begin{proof}
\begin{enumerate}
    \item[(a)] Let us prove that we can assign weights so that the graph becomes geodetic. We do this using mathematical induction on $n$, the number of vertices in the graph. The assertion of the induction coincides with the assertion of the theorem.

Base: $n=2$, two vertices connected by one edge, obviously, assigning the weight 1, the graph will be geodetic.

Induction step: assume the statement is proved for $n$ vertices, we will prove it for $n+1$. We remove some vertex $v$ and edges outgoing from it so that the graph remains connected (for instance, fix a spanning tree of the graph and remove one of its leafs), obtaining a subgraph $H$ of the graph $G$. In the graph $H$ there are $n$ vertices, we assign weights to its edges so that it becomes geodetic (we can do this according to the induction hypothesis). Let the sum of the weights of the edges at this moment be equal to $S$.

Now we assign weights to the edges outgoing from $v$. Let one of the edges $uv$ have weight $S+1 $, and all the others going out from $v$ have weight $2S+2$. Then the graph will be geodetic.

Indeed, if we select the given weights, then between any two vertices other than $v$ the geodesic will not change after adding the vertex $v$ and the edges going out from it. And also geodesics going out from $v$ will always pass through exactly one edge $uv$ with weight $S+1$, and then the geodesic will be the only one since it is a geodesic in the graph $H$ with added edge $uv$. See the example in the Figure \ref{fig: alg_vzvesh}.

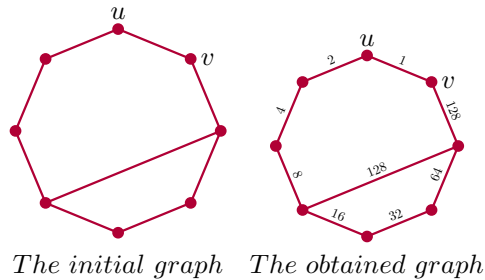
\begin{figure}[H]
\centering
\begin{tikzpicture}[scale=0.9,color=best,line width=0.03cm,
block/.style={circle,fill=best,draw=best,inner sep=1.25pt}]
    \foreach \i in {1,...,8}
    {
        \coordinate (\i) at (\i*360/8:1.5);
        \node[block] at (\i*360/8:1.5) {};
    }
    \begin{scope}[color=best]
\draw (1) --(2);
\draw (3) --(2);
\draw (3) --(4);
\draw (5) --(4);
\draw (5) --(6);
\draw (7) --(6);
\draw (7) --(8);
\draw (1) --(8) -- (5);
    \end{scope}
    \node[color=black] at (-90:2) {$The \ initial \ graph$};
    \node [color=black, anchor=west] at (1) {$v$}; 
    \node [color=black, anchor=south] at (2) {$u$};
\end{tikzpicture}
\begin{tikzpicture}[scale=0.8,color=best,line width=0.03cm,
block/.style={circle,fill=best,draw=best,inner sep=1.25pt}]
    \foreach \i in {1,...,8}
    {
        \coordinate (\i) at (\i*360/8:1.5);
        \node[block] at (\i*360/8:1.5) {};
    }
\begin{scope}[color=best]
\draw (1) -- node [color=black, above, sloped, scale=0.5] {$1$} (2);
\draw (3) -- node [color=black, above, sloped, scale=0.5] {$2$} (2);
\draw (3) -- node [color=black, above, sloped, scale=0.5] {$4$} (4);
\draw (5) -- node [color=black, above, sloped, scale=0.5] {$8$} (4);
\draw (5) -- node [color=black, above, sloped, scale=0.5] {$16$} (6);
\draw (7) --node [color=black, above, sloped, scale=0.5] {$32$}(6);
\draw (7) --node [color=black, above, sloped, scale=0.5] {$64$}(8);
\draw (1) -- node [color=black, above, sloped, scale=0.5] {$128$} (8);
\draw (8) -- node [color=black, above, sloped, scale=0.5] {$128$} (5);
\node [color=black, anchor=west] at (1) {$v$}; 
\node [color=black, anchor=south] at (2) {$u$}; 
    \end{scope}
    \node[color=black] at (-90:2) {$The \ obtained \  graph$};
\end{tikzpicture}
\caption{Constructing weighted geodetic antipodal graphs.}
\label{fig: alg_vzvesh}
\end{figure}

 \item[(b)] Let us prove that the same assignments of weights to the edges from the part $(a)$ also guarantees us that the achieved graph is antipodal. We use the same induction hypothesis, but for the antipodes. We check that the same base for two vertices is also true for antipodal graphs. Then for the induction step we use the same notations as previously in the part $(a)$. Then the antipode of vertices different from $v$ is obviously the vertex $v$, and the antipode of the vertex $v$ is the antipode of the vertex $u$ in the subgraph $H$.
\end{enumerate}
\end{proof}

\begin{theorem} \label{th: vzv geod}
Let $G$ be a connected weighted graph on $n$ vertices. Then it is possible to check whether the graph $G$ is geodetic/antipodal or not for $O(n^3)$. 
\end{theorem}
\begin{proof}
The main idea is Dijkstra's algorithm (similar to building a spanning tree using Breadth First Search).

We will check the graph is geodetic or not at the stages of constructing the spanning tree $T$ of the graph $G$. Let us take as the root of the tree an arbitrary vertex of the graph $G$ vertex $ v_1 $ -- it constitutes the zero tier. Further, add to $T$ all vertices adjacent to $v_1$ (together with edges from $v_1$) -- let, without loss of generality, these are $v_2, v_3, ... v_k$. Let the weight of the edge from $v_1$ to $v_i$ be equal to $k_i$, then we add the vertex $v_i$ to the $k_i$-th tier.

Further, consider the vertex of the topmost tier from $T$, other than $v_1$. Let, without loss of generality, this is the vertex $v_2$.
Now we consider the neighbors of the vertex $v_2$ in the graph $G$. Two options are possible:
\begin{enumerate}
    \item[1)] If some of the neighbors $v_2$ was not added earlier, then add it to the tree $T$ in a tier different from the tier $v_2$ by the number corresponding to the edge between $v_2$ and this vertex.
    \item[2)] If some of the neighbors of $v_2$, let $v_k$, have already been added to the tree $T$, then we look at the weight of the edge $v_2v_k$ and the difference between the tiers $v_2$ and $v_k$. If the weight of the edge is less, then we transfer the vertex to a tier different from the tier $v_2$ by the weight of the edge and remove another edge from $v_k$ so that there are no cycles in $T$. If the weight of the edge is greater, then we do nothing. If the weights of the edges are equal, then we draw this edge (we will have a cycle, if after repeating this operation the cycle remains, then the graph is not geodetic)
\end{enumerate}

After that, the vertices $v_1, v_2$ are already fixed in their tiers.

Next, we do the same with the rest of the vertices. We consider a vertex of the uppermost tier from $T$ that differs from the vertices already considered and fixed in their tiers. Let, without loss of generality, this is the vertex $v$. If two edges go to this vertex $v$ from the upper tiers, then the graph is not geodetic, otherwise we consider the neighbors of the vertex. Two options are possible:
\begin{enumerate}
    \item[1)] If some of the neighbors $v$ was not added earlier, then add it to the tree $T$ in a tier different from the tier $v$ by the number corresponding to the edge between $v$ and the given vertex.
    \item[2)] If some of the neighbors of $v$, let $u$, have already been added to the tree $T$, then look at the weight of the edge $vu$ and the difference between the tiers $v$ and $u$. If the weight of the edge is less, then we transfer the vertex $u$ to a tier that differs from the tier of $v$ by the weight of the edge and remove other edges coming from $u$. If the weight of the edge is greater,  we do nothing. If the weight of the edge is equal to the difference in tiers of $u$ and $v$, then we draw this edge (we will have a cycle).
\end{enumerate}
After that, the vertex $v$ is fixed in its tier.
 
At the end of the algorithm, we will find out whether there is exactly one geodesic from the root to other vertices (if there are no cycles in the tree $T$, each geodesic is unique), and also by the number of vertices in the last tier, we will understand whether the graph is antipodal or not. To construct such an algorithm, we need $O(n^2+m)$, where $m$ is the number of edges in the graph. Repeating this algorithm, choosing other vertices as the root of the tree, we need $O(n^3)$.
\end{proof}

\begin{theorem} \label{th: weighted her}
Let $H$ be an arbitrary weighted graph with integer weights of edges that does not contain vertices between which there are several geodesics of length two. Then there is a weighted geodetic graph $G$ which contains $H$ as an induced subgraph, and the weights of the added edges are equal to one or two.
\end{theorem}
\begin{proof}
 Let us describe the algorithm for constructing the graph $G$.
 \begin{enumerate}
     \item[1)] Take the graph $G$ isomorphic to $H$. 
     \item[2)] Let us fix an arbitrary vertex $v$ from $V(H)$ and construct a bearing tree, described in Theorem \ref{th: vzv geod}, of the graph $G$ with the root $v$. If in some tier below the second there is a vertex $u \in V(H)$, then we add a vertex $t$ to the graph $G$ and draw edges $vt$ and $ut$ with weights equal to one. It is easy to see that in this case there will be no vertices with two geodesics of length two between them. Fixing all vertices of $V(H)$ as the root of the bearing tree one by one, we obtain a graph in which there is exactly one geodesic of length one or two between any two vertices from $V(H)$.
     \item[3)] Let us fix an arbitrary vertex $v$ from $V(G) \setminus V(H)$ as the root of the bearing tree of the graph $G$. Let there be some vertex $u$ in the tier below the third. Then add an edge $uv$ of length one to the graph $G$. It is easy to see that in this case there will be no vertices with two geodesics of length two between any two vertices of $G$, since if the added edge of length one is contained in some such geodesic, then it would mean that the vertex $u$ is in one of the first three tiers. And also multiple edges will not appear, since only edges with weight one emerge from the considered vertex. Further, we consider a bearing tree with the root $v$ of the achieved graph $G$ until all the vertices are in the first three tiers. Then, we will fix all the vertices from $V(G) \setminus V(H)$ as the root of the bearing tree.
     \item[4)] Suppose the obtained graph is not geodetic. Then note that if there are several geodesics between some vertices, then these geodesics are necessarily of length three. Then if there are several geodesics between some two vertices, draw an edge of length two between them. Note that multiple edges cannot appear, since both vertices between which we draw an edge cannot belong to $V(H)$, and edges with weights two and one emerge from the vertices $V(G) \setminus V(H)$ and there are no other edges. After this operation of adding new edge, no pairs of vertices with multiple geodesics of length two between them will appear.
     \item[5)] Repeat step four as long as possible.
\end{enumerate}
The algorithm will stop, since we do not add vertices after step one, so there are finitely many of them. This means that number of pairs of vertices between which there can be a geodesic of length three is also finite.
\end{proof}

\section{Problems}

\subsection{Isomorphic graphs} \label{pa: heriditary}
\noindent

\begin{lemma} [Rasin \cite{husimi}]
Any two cacti can be checked for isomorphism in polynomial time.
\end{lemma}

\begin{open}
Let two geodetic graphs $G$ and $G'$ be given. Is it possible to check whether these graphs are isomorphic or not in polynomial time?
\end{open}

\begin{conjecture}
An invariant for isomorphic geodetic graphs is a bijection from the bearing trees of the graph $G$ to the bearing trees of the graph $G'$ such that the trees corresponding to each other are isomorphic.
\end{conjecture}

Note that if the geodetic graphs $G$ and $G'$ are isomorphic, then their bearing trees will also be isomorphic, which follows from the Corollary \ref{col: uh} below. Generally speaking, this is not true for arbitrary non-geodetic graphs, since the bearing tree can be defined non-uniquely, for example, see Figure \ref{fig: oh}:
\begin{figure}[htb]
\centering
\begin{tikzpicture}[color=best, x=0.5cm,y=0.4cm, line width=0.02cm]
\draw [fill=best] (0,0) circle (2pt);
\draw [fill=best] (2,2) circle (2pt);
\draw [fill=best] (4,0) circle (2pt);
\draw [fill=best] (2,-2) circle (2pt);
\draw [fill=best] (4,-2) circle (2pt);
\draw (4,-2)--(4,0)--(2,2)--(0,0)--(2,-2)--(4,0);
\node [color=black] at (2,-4) {Initial graph};
\end{tikzpicture}
\qquad
\begin{tikzpicture}[color=best, x=0.5cm,y=0.4cm, line width=0.02cm]
\draw [fill=best] (0,0) circle (2pt);
\draw [fill=best] (2,2) circle (2pt);
\draw [fill=best] (4,0) circle (2pt);
\draw [fill=best] (2,-2) circle (2pt);
\draw [fill=best] (4,-2) circle (2pt);
\draw (4,-2)--(4,0)--(2,2)--(0,0)--(2,-2);
\node [color=black] at (2,-4) {Bearing tree 1};
\end{tikzpicture}
\qquad
\begin{tikzpicture}[color=best, x=0.5cm,y=0.4cm, line width=0.02cm]
\draw [fill=best] (0,0) circle (2pt);
\draw [fill=best] (2,2) circle (2pt);
\draw [fill=best] (4,0) circle (2pt);
\draw [fill=best] (2,-2) circle (2pt);
\draw [fill=best] (4,-2) circle (2pt);
\draw (4,-2)--(4,0)--(2,2)--(0,0);
\draw (4,0)--(2,-2);
\node [color=black] at (2,-4) {Bearing tree 2};
\end{tikzpicture}
\caption{Non-unique bearing trees.}
\label{fig: oh}
\end{figure}
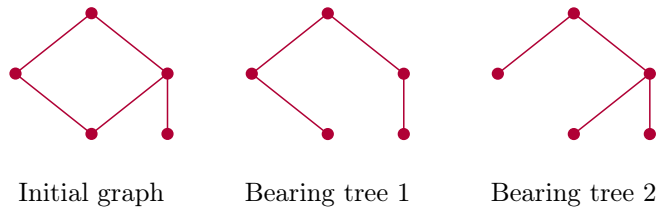

 \begin{corollary} \label{col: uh}
The bearing tree of a geodetic graph with a fixed root is uniquely defined, which follows from the Lemma \nolinebreak\ref{le:monogeo_algo}.
\end{corollary} 

\subsection{Open problems}
We conclude with two open problems:
\begin{itemize}
    
    \item Is it true that for any graph $H$ not containing $C_4$ and $K_4-e$ as an induced subgraph there exists a geodetic graph $G$ such that $H$ is an induced subgraph of the graph $G$?

    The following problem appears in \cite{c_4 free} that was rephrased in terms of graph theory in \cite{projective Hamiltonian}:
    Is every finite bipartite $C_4$-free graph a subgraph of the Levi graph of a finite projective plane?

    So, in particular, it's curious to find out whether there exists a geodetic graph that contains $Levi(\pi_r)$ as an induced subgraph or not. 
    
    \item Let two geodetic graphs $G$ and $G'$ are given. Is it possible to check whether these graphs are isomorphic or not in polynomial time?
\end{itemize}

\vspace{1cm}
\emph{E-mail address:} \texttt{\href{mailto: dimgor2003@gmail.com}{dimgor2003@gmail.com}}

\vspace{0.2cm}
\emph{E-mail address:} \texttt{\href{mailto:david.zmiaikou@gmail.com}{david.zmiaikou@gmail.com}}


\end{justify}
\end{document}